\documentclass[12pt]{amsart}
\usepackage[margin=2cm]{geometry}
\usepackage{graphicx}
\usepackage{ extpfeil }
\usepackage{verbatim}
\usepackage{array}
\usepackage{booktabs, boldline,multirow}
\usepackage{float}
\usepackage{longtable}
\usepackage{tabularx,ragged2e}
\usepackage{mathrsfs,amssymb, euscript, upgreek}
\usepackage[all]{xy}
\usepackage{tikz,tikz-3dplot}
\usetikzlibrary{shadows}
\usepackage{hyperref, url}

\newcommand{\xref}[1]{{\rm \ref{#1}}}
\newcommand{\comp}{\mathbin{\scriptstyle{\circ}}}

\newcommand\mi{%
\setbox0=\hbox{-}%
\vcenter{%
\hrule width\wd0 height \the\fontdimen8\textfont3%
}%
}

\newcommand{\bi}{\boldsymbol{i}}
\newcommand{\bzeta}{\boldsymbol{\zeta}}
\newcommand{\mumu}{{\boldsymbol{\mu}}}

\newcommand{\Aut}{\operatorname{Aut}}
\newcommand{\Sing}{\operatorname{Sing}}
\newcommand{\Supp}{\operatorname{Supp}}
\newcommand{\St}{\operatorname{St}}

\newcommand{\Pic}{\operatorname{Pic}}
\newcommand{\Cl}{\operatorname{Cl}}
\newcommand{\Tr}{\operatorname{Tr}}

\newcommand{\Fix}{\operatorname{Fix}}

\newcommand{\Proj}{\operatorname{Proj}}
\newcommand{\Bs}{\operatorname{Bs}}

\newcommand{\p}{\operatorname{p}_{\mathrm{a}}}
\newcommand{\g}{\operatorname{g}}
\newcommand{\rk}{\operatorname{rk}}
\newcommand{\qq}{\mathbin{\sim_{\scriptscriptstyle{\mathbb{Q}}} } }

\newcommand{\GL}{\operatorname{GL}}
\newcommand{\PGL}{\operatorname{PGL}}
\newcommand{\SL}{\operatorname{SL}}
\newcommand{\SO}{\operatorname{SO}}
\newcommand{\SU}{\operatorname{SU}}

\newcommand{\Lef}{\operatorname{Lef}}
\newcommand{\N}{\operatorname{N}}
\newcommand{\C}{\operatorname{C}}
\newcommand{\Exc}{\operatorname{Exc}}

\newcommand{\type}[1]{$\mathrm{#1}$}

\newcommand{\z}{\operatorname{z}}
\newcommand{\Sy}{\operatorname{S}}
\newcommand{\Syl}{\operatorname{Syl}}
\newcommand{\CC}{{\mathbb C}}
\newcommand{\RR}{{\mathbb R}}
\newcommand{\ZZ}{{\mathbb Z}}
\newcommand{\QQ}{{\mathbb Q}}
\newcommand{\PP}{{\mathbb P}}

\newcommand{\Aff}{\mathbb{A}}
\newcommand{\FF}{{\mathbb F}}

\newcommand{\bF}{\mathbf{F}}

\newcommand{\Points}{\Xi}
\newcommand{\Lines}{\mathbf{L}}
\newcommand{\DLines}{\mathcal{L}}
\newcommand{\DCurves}{\mathcal{R}}
\newcommand{\Disc}{\updelta}

\newcommand{\Alt}{{\mathfrak A}}
\newcommand{\Sym}{{\mathfrak S}}
\newcommand{\Dih}{{\mathfrak D}}

\newcommand{\OOO}{{\mathscr{O}}}
\newcommand{\NNN}{{\mathscr{N}}}
\newcommand{\PPP}{\mathscr{P}}

\newcommand{\LLL}{\mathscr{L}}

\def\acts{\curvearrowright}

\newcounter{NN}
\counterwithin{NN}{table}
\renewcommand{\theNN}{{\rm\arabic{NN}${}^o$}}
\newcommand{\nr}{\refstepcounter{NN}\theNN}

\renewcommand\labelenumi{\rm (\roman{enumi})}
\renewcommand\theenumi{\rm (\roman{enumi})}

\makeatletter
\@addtoreset{equation}{subsection}
\makeatother

\swapnumbers

\theoremstyle{plain}
\newtheorem{theorem}[subsection]{Theorem}
\newtheorem{lemma}[subsection]{Lemma}

\newtheorem{proposition}[subsection]{Proposition}
\newtheorem{proposition-definition}[subsection]{Proposition-Definition}
\newtheorem{stheorem}[equation]{Theorem}
\newtheorem{stheorem-definition}[equation]{Theorem-Definition}

\newtheorem{scorollary}[equation]{Corollary}
\newtheorem*{claim*}{Claim}

\newtheorem{sclaim}[equation]{Claim}
\newtheorem{slemma}[equation]{Lemma}
\newtheorem{sproposition}[equation]{Proposition}

\theoremstyle{definition}

\newtheorem*{definition*}{Definition}

\newtheorem{example-remark}[subsection]{Remark-Example}
\newtheorem{subexample-remark}[equation]{Remark-Example}

\newtheorem{notation}[subsection]{Notation}

\newtheorem*{notation*}{Notation}
\newtheorem{snotation}[equation]{Notation}

\newtheorem{sremark}[equation]{Remark}

\begin{document}

\title{Icosahedron in birational geometry}

\address{
Steklov Mathematical Institute of Russian Academy of Sciences, Moscow, Russian Federation
}
\email{prokhoro@mi-ras.ru}

\author{Yuri~Prokhorov}
\thanks{
This work was supported by the Russian Science Foundation under grant no. 23-11-00033, 
\url{https://rscf.ru/project/23-11-00033/}}

\begin{abstract}
 We study quotients of projective and affine spaces by various actions of the icosahedral group.
Basically we concentrate on the rationality questions.
\end{abstract}

\maketitle

\section{Introduction}
The (regular) icosahedron is one of the five Platonic solids that has 
$20$ triangular faces, $30$ edges, and $12$ vertices (see Fig.~\ref{fig:ico}).
The group of orientation preserving symmetries of the icosahedron is isomorphic to
the alternating group $\Alt_5$ and the full
group of symmetries
is isomorphic to direct product $\Alt_5\times \{\pm 1\}$. 

The study of the icosahedral group $\Alt_5$ has a long history; this group plays a role in various branches of mathematics.
Here we just mention that it is closely connected with many other remarkable objects: the golden ratio, the quaternions
\cite{Conway-Derek:book}, the quintic equation \cite{Klein1956}, a convex regular $4$-polytope called the $600$-cell \cite{Coxeter:polytopes}, and a manifold called the Poincar\'e homology 
$3$-sphere 
\cite{Kirby-Scharlemann:Poincare-homology-3-sphere}.

In this paper we discuss only one aspect of the rich geometry related to icosahedron: 
the geometry of the certain quotients by icosahedral group and groups related to it.
In particular, we explicitly describe the quotients of so-called
$\Alt_5$-minimal rational surfaces by $\Alt_5$.
Recall the following fact that 
can be extracted from the classification {\cite{Dolgachev-Iskovskikh}}, see also {\cite[Theorem~B.10]{Cheltsov:2ineq}}.
\begin{theorem}
\label{thm:Cr2}
Suppose that $\Alt_5$ effectively act on a rational surface $Y$.
Then there exists an $\Alt_5$-equivariant birational map $Y \dashrightarrow S$, where 
$S$ is either the projective plane $\PP^2$, the quadric $\PP^1\times \PP^1$, or the quintic del Pezzo surface $S_5\subset \PP^5$.
\end{theorem}

In Sections \ref{sect:P2}, \ref{Sect:P1P1}, and~\ref{sect:DP5} we study the quotients of
$\PP^2$, $\PP^1\times \PP^1$ and $S_5$
by $\Alt_5$ and by some related groups.
Using these computations we give short conceptual proofs of the rationality of the quotients of 
the three-dimensional projective space by two different actions of $\Alt_5$ (see Theorems~\ref{th-SL25} and~\ref{thm:C4/A5}).
These facts were known earlier, see \cite{Kolpakov-Prokhorov-1992}, \cite{Kervaire-Vust}, \cite{Maeda-T-1989}.
In Section
\ref{sect:V6} we prove the rationality of 
the quotient of the 
projectivization of the faithful irreducible six-dimensional representation of~$\tilde \Alt_5$.
Finally, in Section \ref{sect:V5}, we prove the following result.

\begin{theorem}
\label{thm:main}
Let $V$ be any representation of the binary icosahedral group $\tilde \Alt_5$.
Then the variety $V/\tilde \Alt_5$ is rational.
\end{theorem}

The rationality question of quotients by finite linear groups
is interesting in relation to famous Emmy Noether's problem \cite{Noether_1918}.
The problem still is not solved in its generality (see surveys \cite{colliotthelene-Sansuc-2005}, \cite{P:Inv10}).
Another motivation for rationality problem is provided by \textit{moduli spaces}. 
Birationally, they are often representable as quotients of certain
simple varieties 
by actions of linear algebraic groups (see e.g. \cite{Shepherd-Barron-S5}).

Note that many results of the paper were known earlier. Our aim was to systematize them 
and provide uniform proofs.

\subsection*{Acknowledgements.} The author would like to thank Alexander Kuznetsov and Constantin Shramov for useful discussions.
Some computations were performed by using Macaulay2 \cite{M2}.

\tdplotsetmaincoords{60}{100}
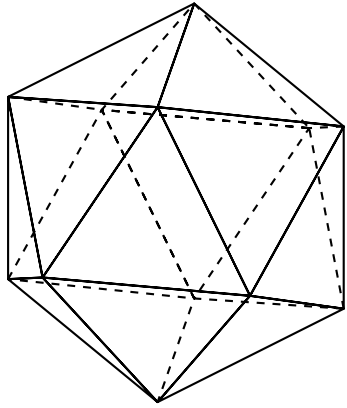
\begin{figure}[H]
\begin{tikzpicture}[tdplot_main_coords,scale=0.7,line join=round]
\pgfmathsetmacro\a{2}
\pgfmathsetmacro{\phi}{\a*(1+sqrt(5))/2}
\path 
coordinate(A) at (0,\phi,\a)
coordinate(B) at (0,\phi,-\a)
coordinate(C) at (0,-\phi,\a)
coordinate(D) at (0,-\phi,-\a)
coordinate(E) at (\a,0,\phi)
coordinate(F) at (\a,0,-\phi)
coordinate(G) at (-\a,0,\phi)
coordinate(H) at (-\a,0,-\phi)
coordinate(I) at (\phi,\a,0)
coordinate(J) at (\phi,-\a,0)
coordinate(K) at (-\phi,\a,0)
coordinate(L) at (-\phi,-\a,0); 
\draw[dashed, thick] (B) -- (H) -- (F) 
(D) -- (L) -- (H) --cycle 
(K) -- (L) -- (H) --cycle
(K) -- (L) -- (G) --cycle
(C) -- (L) (B)--(K) (A)--(K)
;

\draw[thick]
(A) -- (I) -- (B) --cycle 
(F) -- (I) -- (B) --cycle 
(F) -- (I) -- (J) --cycle
(F) -- (D) -- (J) --cycle
(C) -- (D) -- (J) --cycle
(C) -- (E) -- (J) --cycle
(I) -- (E) -- (J) --cycle
(I) -- (E) -- (A) --cycle
(G) -- (E) -- (A) --cycle
(G) -- (E) -- (C) --cycle
;
\end{tikzpicture}
\label{fig:ico}
\caption{Icosahedron}
\end{figure}

\section{Preliminaries}
\label{sect:pre}
\begin{notation}

We work over the field $\CC$ of complex numbers.
The following standard notation will be used throughout the paper.
\begin{itemize}

\item 
$\mumu_n$ denotes the multiplicative cyclic group of order $n$;
\item 
$\Sym_n$ (resp. $\Alt_n$) denotes the symmetric (resp. alternating) group of degree $n$; 
\item 
$\Dih_n$ denotes the dihedral group of order $2n$;
\item 
$\N_G(H)$ denotes the normalizer of a subgroup $H$ in $G$;

\item 
$\PP(w_1,\dots,w_n)$ is the weighted projective space with weights $w_1,\dots,w_n$ \cite{Dolgachev:WPS};

\item 
$\upsilon_d$ denotes the $d$-uple Veronese embedding $\upsilon_d: \PP^k \hookrightarrow \PP^{N}$, $N=\binom{k+d}{d}-1$.

\item 
If a group $G$ acts on a variety $X$, then $\Fix(G,X)$ denotes the set 
of $G$-fixed points and $\Fix(G,X)^{(k)}$ denotes the union of all components of $\Fix(G,X)$
of dimension $k$. By $\St_G(x)$ we denote the stabilizer subgroup of $x\in X$.

\end{itemize}

The icosahedral group $\Alt_5$ has a unique non-trivial extension $\tilde \Alt_5$ by the cyclic group of order $2$
which is called \textit{binary icosahedral group}; we denote it by $\tilde \Alt_5$. It can be constructed as
the preimage of $\Alt_5\subset \SO_3(\RR)$ under the covering homomorphism $\SU_2(\CC)\to \SO_3(\RR)$.
There is a well-known isomorphism $\tilde{\Alt}_5\simeq\SL_2(\bF_5)$, where $\bF_q$ the finite field of $q$ elements.

\end{notation}

When we say that $V$ is a \textit{representation} of a group $G$, we usually mean 
that $V$ is a $\CC[G]$-module and the homomorphism $G\to \GL(V)$ is supposed to be fixed.
The irreducible representations of $\Alt_5$ (resp. $\Sym_5$) will be denoted by $V_1$,
$V_3$, $V_3'$, $V_4'$, $V_5$ (resp. $W_1$, $W_1'$, $W_4$, $W_4'$, $W_5$, $W_5'$, $W_6$), where the subscript is the degree.
Similarly, the irreducible faithful representations of $\tilde \Alt_5$
will be denoted by $V_2$,
$V_2'$, $V_4$, $V_6$.
For the convenience of the reader, the character tables are reproduced in the appendix.

\subsection{Conic bundles}
General facts about conic bundles can be found in \cite{P:rat-cb:e}.

For a standard conic bundle $\varphi: Y\to S$ by $\Delta_{\varphi}\subset S$ we denote the discriminant divisor.

\begin{stheorem}[{\cite{Sarkisov:82e}}]
\label{thm:Sar}
Let $f: X\to S$ be a rational curve fibration and let $\Lambda\subset S$ be 
a closed subset such that $f$ is smooth over $S\setminus\Lambda$.
Then there exists 
the following commutative diagram
\[
\xymatrix@R=1.2em{
X\ar[d]_{f}& \tilde X \ar@{-->}[l]_<(+.3){\gamma} \ar[d]^{\tilde f}
\\
S& \tilde S\ar[l]_{\beta}
}
\]
where $\tilde f$ is a standard conic bundle,
$\gamma$ is a birational map, and $\beta$ is a birational morphism
that is an isomorphism over $S\setminus\Lambda$. For the discriminant divisor 
$\Delta_{\tilde{f}}\subset \tilde S$ we have $\beta(\Delta_{\tilde{f}})\subset 
\Lambda$.
\end{stheorem}

\begin{stheorem}[{see e.g. \cite[Theorem~1]{Iskovskikh:Duke} or \cite[Proposition~5.6]{P:rat-cb:e}}]
\label{thm:rat}
Let $\varphi: Y\to Z$ be a standard conic bundle over a rational surface. Assume that on $Z$ there exists a base point free 
pencil of rational curves $\LLL$ such that $\LLL\cdot \Delta_{\varphi}\le 3$. Then $X$ is rational.
\end{stheorem}

\begin{sproposition}
\label{prop:cb}
Let $\varphi: Y\to Z$ be a standard conic bundle over a surface. 
Suppose that a component $\Delta_1\subset \Delta_{\varphi}$ is a smooth rational curve.
Then $\Delta_1\cdot (\Delta_{\varphi}-\Delta_1)$ is even and not zero. 
\end{sproposition}

\begin{proof}
Let $\hat \Delta_1$ be the component Hilbert scheme that parametrizes  the components of  degenerate fibers over $\Delta_1$.
Then there is a natural double cover $\hat \Delta_1\to \Delta_1$ branched over the points $\Delta_1\cap \overline{(\Delta_{\varphi}\setminus \Delta_1)}$.
Since $\Delta_1$ is a smooth rational curve this number must be even. 
If $\Delta_1\cap \overline{(\Delta_{\varphi}\setminus \Delta_1)}=\varnothing$, then the double cover splits. 
This means that the inverse image $\varphi^{-1}(\Delta_1)$ is reducible.
On the other hand, the relative Picard number $\uprho(Y/Z)$ equals $1$ because $\varphi$ is a 
standard conic bundle (see \cite{P:rat-cb:e}).
\end{proof}

\begin{slemma}
\label{lemma:P2}
Let 
$f: X\to \PP^2$ be a rational curve fibration such that $f$ is smooth over $\PP^2\setminus D$,
where $D$ is the union of a smooth conic $C$ and three distinct lines $M_1,\, M_2,\, M_3$ touching $C$ \textup(see Fig.~\xref{fig:ell}\textup). Then $X$ is rational.
\end{slemma}

\begin{figure}[h]
\begin{tikzpicture}[xscale = .8, yscale = .8]
\draw[black, thick] {(0,0) ellipse (5em and 3.5em)} node[ xshift=35, yshift=10] {$C$};
\draw[black, thick] (-5,1.47) -- (5,1.47)node[above, xshift=-75] {$M_1$};
\draw[black, thick] (-4,2) -- (0.2,-3.2)node[above, yshift=20, xshift=25] {$M_3$};
\draw[black, thick] (4,2) -- (-0.2,-3.2)node[above, yshift=20, xshift=-25] {$M_2$};
\fill (0,1.47) circle (1.5pt)node[above, xshift=-5] {$Q_1$};
\fill (-1.84,-0.65) circle (1.5pt)node[left, xshift=-5] {$Q_2$};
\fill (1.84,-0.65) circle (1.5pt)node[right, xshift=5] {$Q_3$};
\end{tikzpicture} 
\caption{}
\label{fig:ell}
\end{figure}

\begin{proof}
Apply Theorem~\ref{thm:Sar} with $S=\PP^2$. 
Let $\LLL$ be the pencil of lines passing through $Q_1$ and let $\tilde \LLL$ be the proper transform of $\LLL$ 
on $\tilde S$. 
For a general member $L\in \LLL$ the intersection $L\cap D$ consists of $ Q_1$ and three more distinct points.
We may assume that 
there is a decomposition 
\[
\beta: \tilde S \overset{\beta_2}\longrightarrow 
S_1\overset{\beta_1}\longrightarrow S=\PP^2, 
\]
where
$\beta_1$ is the blowup of $ Q_1$. Then $\tilde 
\LLL$ is base point free. Let $\tilde E\subset \tilde S$ be the proper transform 
of the $\beta_1$-exceptional divisor.
Then $\tilde E$ meets the set
$\overline{\beta^{-1}(D)\setminus \tilde E}$
at one point, say $\tilde P_1$, and a general member $\tilde L\in \tilde \LLL$
does not pass through $\tilde P_1$. In this case, $\tilde E$ is not contained in the discriminant curve $\Delta_{\tilde{f}}$ of $\tilde f$
(see Proposition~\ref{prop:cb}).
Therefore, $\Delta_{\tilde{f}}\cdot \tilde L\le 3$ and $X$ is rational by Theorem~\ref{thm:rat}.
\end{proof}

\begin{slemma}
\label{lemma:P(3,4,5)}
Let $S$ be the weighted projective plane $\PP(3,4,5)$ and let $P_3:=(1,0,0)$, 
$P_4:=(0,1,0)$, $P_5:=(0,0,1)$ be its toric points.
Let
$f: X\to S$ be a rational curve fibration such that $f$ is smooth over 
$S\setminus (D\cup \Sing(S))$,
where $D$ is an element of the linear system $|\OOO_S(20)|$. 
Assume that $P_4,\, P_5\notin D$, and the pair $(S,D)$ is purely log terminal (plt) at $P_3$.
Then $X$ is rational.
\end{slemma}

\begin{proof}
Apply Theorem~\ref{thm:Sar}.
Then $\beta(\Delta_{\tilde{f}})\subset D\cup \Sing(S)$.
Note however that $\beta^{-1}(P_i)$, $i=3,4,5$ is a tree of smooth rational 
curves,
$P_4,\, P_5\notin D$, and the pair $(S,D)$ is plt at $P_3$ by 
\eqref{eq:Psi:00a}.
The latter implies that $\beta^{-1}(P_3)$ meets the proper transform of $D$ 
transversally at one point.
Hence $\Delta_{\tilde{f}}\subset \beta^{-1}\big(D\big)$ and none of the 
components of $\beta^{-1}\big(\Sing(S)\big)$ 
nor $\beta$-exceptional divisors over smooth points of $D$ 
are contained in $\Delta_{\tilde{f}}$ (see Proposition~\ref{prop:cb}).
In particular, $\Delta_{\tilde{f}}\subset \Supp(\lfloor f^*D\rfloor)$.

Consider the pencil of rational curves $\LLL=|\OOO_S(8)|$. Any member $L\in 
\LLL$ is given by
$\lambda_0 z_3z_5+\lambda_1 z_4^2=0$. The pencil $\LLL$ has two base points at 
$P_3$ and $P_5$. We may assume that the proper transform $\tilde \LLL$ of $\LLL$ 
on $\tilde S$ is base point free. 
Since $\Delta_{\tilde{f}}\subset \Supp(\lfloor f^*D\rfloor)$, we have 
\[
\textstyle
\Delta_{\tilde{f}}\cdot \tilde \LLL\le f^*D\cdot \tilde \LLL\le D\cdot \LLL= 
\frac {20\cdot 8 }{3\cdot 4\cdot 5}=\frac 83<3.
\]
Then $X$ is rational by Theorem~\ref{thm:rat}.
\end{proof}

\begin{slemma}
\label{lemma:P(1,3,5)}
Let $S:=\PP(1,3,5)$ and let
$f: X\to S$ be a rational curve fibration such that $f$ is smooth over 
$S\setminus (D\cup \Sing(S))$,
where $D$ is an element of the linear system $|\OOO_S(15)|$ that does not pass 
through singular points of $S$. Then $X$ is rational.
\end{slemma}

\begin{proof}
Similar to the proof of Lemma~\ref {lemma:P(3,4,5)}. Here we take 
$\LLL=|\OOO_S(3)|$. Then any member $L\in \LLL$ is given by
$\lambda_0 z_3+\lambda_1 z_1^3=0$ and the base locus of $\LLL$ is the single 
point $(0,0,1)$, which does not lie on $D$,
and $D\cdot \LLL=3$. 
\end{proof}

\section{The quotient $\PP(V_3)/\Alt_5$}
\label{sect:P2}

\begin{proposition}
\label{prop:V3}
Let $V$ be an irreducible three-dimensional representation of $\Alt_5$. 
Then the quotient $\PP(V)/\Alt_5$ is isomorphic to the weighted projective plane 
$\PP(1,3,5)$. The branch curve $D\subset \PP(1,3,5)$ is irreducible, has degree $15$, and does not pass through 
singular points of $\PP(1,3,5)$. The singularities of $D$ are as follows: 
two cuspidal points of type \type{A_{4}}, a cuspidal point of type \type{A_{2}} \textup(simple cusp\textup),
and an ordinary double point.
\end{proposition}
\begin{proof}[Outline of the proof]
The group $\Alt_5$ has two irreducible three-dimensional representations $V_3$ and $V_3'$
(see Table~\ref{tab:repSL25}).
They are differ in an outer automorphism of $\Alt_5$, hence studying the quotients,
we can restrict ourselves with one of $V_3$ and $V_3'$, say $V_3$.
Denote $S:=\PP(V_3)/\Alt_5$.
Consider the group $G:=\Alt_5\times\mumu_2$ and let $\mumu_2$ act on $V_3$ by scalar multiplication.
Thus $V_3$ can be regarded as a $G$-module.
The group
$G:=\Alt_5\times\mumu_2$ is known as the Coxeter group \type{H_3} \cite{Bourbaki:Lie:4-6}.
It is isomorphic to the symmetry group of the icosahedron in $\RR^3$,
while the rotation symmetry group is just $\Alt_5\subset G$.
Complexifying the embedding $G\hookrightarrow \GL_3(\RR)$ we obtain a complex 
irreducible representation and we may assume that it is $V_3$.
In this presentation, $G$ is generated by reflections and by the Chevalley–Shephard–Todd theorem
the ring of invariants is generated by three algebraically independent polynomials.
Moreover, these polynomials have degrees $2$, $6$, and $10$ \cite{Bourbaki:Lie:4-6}.
Denote them by $z_2(x_0,x_1,x_2)$, $z_6(x_0,x_1,x_2)$, $z_{10}(x_0,x_1,x_2)$,
where $x_0,x_1,x_2$ are coordinates in $V_3$. Then $\CC[x_0,x_1,x_2]^{G}= \CC[z_2, z_6, z_{10}]$
and so
\[
\PP(V_3)/G\simeq \Proj \CC[z_2, z_6, z_{10}]\simeq \PP(2,6,10)\simeq \PP(1,3,5).
\]

Here is a geometric description of the invariants $z_2$, $z_6$, $z_{10}$.
Since $V_3$
is defined
over $\RR$,
on $V_3$ there exists an invariant non-degenerate quadratic form
$z_2(x_0,x_1,x_2)$.
It defines an invariant smooth conic $C\subset \PP^2=\PP(V)$. Each subgroup
$H\subset \Alt_5$
of order $2$ has two fixed points $P$ and $P'$ on $C\simeq \PP^1$, and the line
passing through $P$ and $P'$
is $H$-invariant. Now, the $15$ order $2$ subgroups $H_i\in \Alt_5$ define
$15$ distinct lines
$L_1,\dots, L_{15}\subset \PP^2$ so that the union $L_1\cup \cdots\cup L_{15}$
is
$\Alt_5$-invariant. Therefore, there exists an invariant polynomial
$z_{15}(x_0,x_1,x_2)$ of degree $15$.
Similarly, considering subgroups in $\Alt_5$ of order $3$ and $5$ we obtain
invariants $z_{10} (x_0,x_1,x_2)$ and $z_{6} (x_0,x_1,x_2)$ of degrees $10$ and
$6$, respectively. Note that the invariants $z_{6}$, $z_{10}$, and $z_{15}$
can be interpreted in terms of the 
icosahedron: they correspond to pars of opposite vertices, faces, and edges, respectively (see~Fig.~\ref{fig:ico}).

Since $z_6$ and $z_{10}$ define different finite sets of points on $C$, we see
that $z_6\notin \CC[z_2]$ and $z_{10}\notin \CC[z_2, z_6]$.
Since $\CC[x_0,x_1,x_2]^{G}$ is generated by three polynomials of degrees $2$, $6$, $10$ (see \cite{Bourbaki:Lie:4-6}), we obtain
that 
$z_2$, $z_6$,
and $z_{10}$
are algebraically independent and $\CC[x_0,x_1,x_2]^{G}=\CC[z_2,z_6,z_{10}]$.

On the other hand, a polynomial of odd degree cannot be invariant
with
respect to $\{\pm E\}$.
Therefore, $z_{15}\notin \CC[x_0,x_1,x_2]^{G}$
and so $\CC[x_0,x_1,x_2]^{\Alt_5}=\CC[z_2, z_6, z_{10}, z_{15}]$
(see \cite[Ch 17, \S 266]{Burnside:1897}). Moreover, $z_{15}^2$ is an invariant
of
$G$, hence in $\CC[x_0,x_1,x_2]^{\Alt_5}$ there is a
relation of the form
\begin{equation}
\label{eq:hyp}
z_{15}^2=\Phi (z_2, z_6, z_{10}),
\end{equation}
where $\Phi$ is a quasi-homogeneous polynomial of weighted degree $30$.
Therefore,
$S$ is the weighted hypersurface given by \eqref{eq:hyp} in $\PP(2,6,10, 15)$.

The branch divisor $D\subset S$ of the quotient morphism
$\pi: \PP^2\to S$ is the image of the union of the reflection lines
$L_i\subset \PP^2$, i.e. the lines of fixed points with respect to the
involutions in $\Alt_5$.
In other words, $D =\pi(\cup_{i=1}^{15} L_i)$, i.e. it is the image of the 
curve in $\PP^2$ given by
$z_{15}=0$. Note that the projection $S\to \PP(2,6,10)\simeq
\PP(1,3,5)$ is an isomorphism, so $S\simeq \PP(1,3,5)$, where the curve $D$ is given by
$\Phi=0$.

Now we present 
explicit formulas for $z_2$, $z_6$, $z_{10}$, $z_{15}$ following \cite[\S 125]{Miller-Blichfeldt-Dickson}
and \cite[Part~II, Ch.~IV, \S~3]{Klein1956}.
First, we may assume that
\[
z_2=x_{0}^{2}+x_{1}x_{2}.
\]
Next, we can take
\[
z_6=8\,x_{0}^{4}x_{1}x_{2}-x_{0}x_{1}^{5}-2\,x_{0}^{2}x_{1}^{2}x_{2}^{2}+x_{1}^
{3}x_{2}^{3}-x_{0}x_{2}^{5}.
\]
Note that this $z_6$ is irreducible (but $z_6- z_2^3$ is represented as a 
product of linear forms as above). 
Now put
\[
z_{10}=-\frac 1{25}\left(256z_2^5-\overline{\mathrm{H}}(z_6,z_2)-480 z_2^2z_6\right),
\]
where $\overline{\mathrm{H}}(z_6,z_2)$ is the \textit{bordered Hessian} of $z_6$ and 
$z_2$ \cite{Miller-Blichfeldt-Dickson}. Finally, 
\[
z_{15}=\frac 1{10} \mathrm{J}(z_6,z_2,z_{10}),
\]
where $\mathrm{J}(z_6,z_2,z_{10})$ is the Jacobian of $z_6$, $z_2$, and 
$z_{10}$. The explicit form of $\Phi$ in the relation \eqref{eq:hyp} has the following form
(see \cite[Part~II, Ch.~IV, \S~3 (16)]{Klein1956}):
\begin{equation}
\label{eq:Phi}
\Phi=-1728z_6^5+z_{10}^3+720z_2z_6^3z_{10}-80z_2^2z_6z_{10}^2+64z_2^3(5z_6^2-z_2z_{10})^2.
\end{equation}
It is easy to see from \eqref{eq:Phi} that $D$ does not pass through singular points of $\PP(1,3,5)$, 
the line $z_2=0$ meets $D$ transversally at one point and this point is smooth on $D$.
Consider the affine chart $U=\{z_2\neq 0\}\simeq \Aff^2$. Put $y_3:=z_6/z_2^3$ 
and $y_5:=z_{10}/z_2^5$.
The equation of $D$ takes the form
\[
64\left(5 y_3^{2}-y_5\right)^{2}-1 728 y_3^{5}+720 y_3^{3}y_5-80 y_3y_5^{2}+y_5^{3}=0.
\]
Then it is easy to check that $D$ is singular at the origin $(0,0)$ and this 
point is cuspidal of 
type \type{A_{4}}. 
Other singularities of $D$ are an
ordinary double point at $( 1, 4)$ and a simple cusp at $(32/27 , 1024/81)$.
Since $D$ lies in the smooth locus of $\PP(1,3,5)$, its arithmetic genus can be computed by the usual
adjunction formula: $\p(D)=4$. This shows that $D$ has no other singularities.
\end{proof}

\begin{sremark}
Any curve in the pencil $C_{\lambda}:=\{z_6=\lambda z_2^3\}\subset \PP(V_3)$ is
$\Alt_5$-invariant. For $\lambda\neq 1,\, \infty$
such a curve
is reduced, irreducible and has six ordinary double singularities at the points 
of the $\Alt_5$-orbit of $(1,0,0)$. These points are in general position, so
by blowing them up we obtain a special cubic surface, called Clebsch cubic (see below).
The normalization $C'_{\lambda}$ of $C_{\lambda}$ is a smooth curve of genus $4$
admitting an effective action of $\Alt_5$ (see~\ref{subs:Bring}). 

A certain (non-homogeneous) invariant of degree $6$ defines a hypersurface in $\PP^3$ 
having $65$ isolated singularities \cite{Barth:icosahedron}. This is the maximum number of singularities that
a sextic hypersurface in $\PP^3$ can have. Similarly, there is an invariant of degree $10$ 
defining the hypersurface with $345$ isolated singularities \cite{Barth:icosahedron}. 
\end{sremark}

\subsection{Clebsch diagonal cubic}
The \textit{Clebsch diagonal cubic surface} is the smooth surface 
$S_{\mathrm{C}}\subset \PP^3\subset \PP^4$ given by 
the following equations:
\[
\sum_{i=1}^5 x_i= \sum_{i=1}^5 x_i^3=0. 
\]
The automorphism group of $S_{\mathrm{C}}$ is the symmetric group $\Sym_5$, acting by permutations of the coordinates. 
Up to isomorphism, $S_{\mathrm{C}}$ is the only 
cubic surface with this automorphism group. 
Another realization of the Clebsch cubic is the blowup $\sigma: S_{\mathrm{C}}\to \PP^2$ at six points corresponding to 
the long diagonals of the icosahedron (see Fig.~\ref{fig:ico}). 
However, this birational morphism is not $\Sym_5$-equivariant.
In fact, on $S_{\mathrm{C}}$ there are two such morphisms 
$\sigma: S_{\mathrm{C}}\to \PP^2$ and $\sigma': S_{\mathrm{C}}\to \PP^2$ switched by the elements of $\Sym_5\setminus \Alt_5$, both of them are $\Alt_5$-equivariant,
and  each of them  corresponds to an invariant double-six \cite[Ch.~9]{Dolgachev-ClassicalAlgGeom}.

\subsection{Bring’s curve}
\label{subs:Bring}
It follows from the Hurwitz formula that the largest possible automorphism group of a genus $4$ curve 
is the symmetric group $\Sym_5$. Moreover, up to isomorphism there exists 
exactly one curve $C_{\mathrm{B}}$ to achieve such an automorphism group \cite{Wiman:S5}. 
It is called \textit{Bring's curve} \cite{Klein1956}, \cite{Edge:Bring}.
Bring's curve $C_{\mathrm{B}}$ is not hyperelliptic and its canonical model is given by 
the following equations in $\PP^3\subset \PP^4$:
\[
\sum_{i=1}^5 x_i= \sum_{i=1}^5 x_i^2= \sum_{i=1}^5 x_i^3=0. 
\]
Clearly, $C_{\mathrm{B}}$ lies on the smooth quadric $Q=\big\{\sum x_i= \sum x_i^2=0\big\}\simeq \PP^1\times \PP^1$
as a divisor of bidegree $(3,3)$, hence $C_{\mathrm{B}}$ is trigonal and admits two $\Alt_5$-equivariant triple covers $C_{\mathrm{B}}\to \PP^1$ that are branched over $\Alt_5$-orbits of length 
$12$. On the other hand, $C_{\mathrm{B}}$ lies on the Clebsch cubic surface $S_{\mathrm{C}}$
and the image of $C_{\mathrm{B}}$ under the $\Alt_5$-equivariant contraction $\sigma: S_{\mathrm{C}}\to \PP^2$
is a member of the pencil $C_{\lambda}:=\{z_6=\lambda z_2^3\}$.
It is an interesting fact is that
the Clebsch diagonal cubic $S_{\mathrm{C}}$ and the Bring's curve $C_{\mathrm{B}}$ 
naturally appear in the theory of modular forms
(see \cite{Hircebruh:UMN}). 

\subsection{Application: rationality of $\PP(V_4)/\Alt_5$}
As an application of Proposition \ref{prop:V3}
we give a complete proof of the following.
\begin{stheorem}[\cite{Kolpakov-Prokhorov-1992}]
\label{th-SL25}
The variety $\PP(V_4)/\Alt_5$ is rational.
\end{stheorem}

\begin{proof}
Since $V_4=\Sy^3 (V_2)$ (see Table~\ref{tab:repSL25}), there exists
an invariant twisted cubic curve $C=\upsilon_3(\PP^1)\subset \PP(V_4))$.\footnote{In fact there are two such
curves, see the proof of Theorem~\ref{th-tildeS5}. }
Let $\sigma \colon Y\to \PP(V_4)$ be the blowup of $C$.
Since $C$ is an intersection of quadrics, the linear system
$|2\sigma^*H-E|$ is base point free and defines a morphism $\varphi: Y\to \PP^2$,
where $H$ is the hyperplane class and $E$ is the exceptional divisor.
In other words, $\varphi$ is
given by the birational transform of the linear system of quadrics
passing through $C$ and the fibers of $\varphi$ are proper transforms of
$2$-secant lines of $C$. 
Note that $Y$ is a
Fano threefold and $\varphi$ is a Mori extremal contraction.
The action $\Alt_5$ on $\PP^2$ is
induced by a faithful representation of $\Alt_5$. Thus we may assume that
$\PP^2=\PP(V_3)$.
We obtain the following
$\Alt_5$-equivariant diagram:
\[
\vcenter{
\xymatrix{
&Y\ar[dl]_{\scriptstyle{\sigma}}\ar[dr]^{\scriptstyle{\varphi}}&
\\
\PP(V_4)\ar@{-->}[rr]&&\PP(V_3)
}}
\]
where $\varphi$ is a $\PP^1$-bundle.
Thus the quotient $\PP(V_4)/\Alt_5$ is birational to $Y/\Alt_5$ and there is a rational curve fibration
$\bar\varphi\colon Y/\Alt_5\to \PP(V_3)/\Alt_5\simeq \PP(1,3,5)$ (see Proposition~\ref{prop:V3}).
The fibration $\bar\varphi$ is smooth over $\PP(1,3,5)\setminus D$, where $D$ is given by $\Phi=0$ (see \eqref{eq:Phi}).
Then by Lemma~\ref{lemma:P(1,3,5)} the variety $Y/\Alt_5$ is rational.
\end{proof}

\section{The quotient $\PP(V_4')/\Alt_5$}
\label{sect:V4prime}
In this section we describe the geometry of the quotient
$\PP(V_4')/\Alt_5$, where $V_4'$ is the irreducible four-dimensional representation of $\Alt_5$.
As a byproduct we give a short conceptual proof of rationality of $\PP(V_4')/\Alt_5$
(Theorem~\ref{thm:C4/A5}).
The earlier proofs can be found in \cite{Kervaire-Vust} and \cite{Maeda-T-1989}.

\begin{theorem}
\label{thm:V4prime}
The variety $\PP(V_4')/\Alt_5$ is a Fano threefold with
$\QQ$-factorial
canonical singularities and $\uprho(X)=1$.
Moreover, the following assertions hold.
\begin{enumerate}
\item
\label{thm:V4prime1}
$(-K_X)^3=16/15$;
\item
\label{thm:V4prime2}
$\Cl(X)\simeq \ZZ$;
\item
\label{thm:V4prime3}
$-K_X\sim 4 A_X$, where $A_X$ is a primitive element of $\Cl(X)$;
\item
\label{thm:V4prime4}
the anticanonical linear system $|-K_X|$
is the pencil given by $\mu z_4+\lambda z_2^2=0$ and for a general member $S\in |-K_X|$ the pair $(X,S)$ is canonical; 
\item
\label{thm:V4Sing-G}
the set of Gorenstein singular
points of $X$ consists of three rational curves $\ell_2$, $\ell_2'$, and $\ell_3$
so that the singularities at the general point of $\ell_2$ and $\ell_2'$ are of type \type{A_1} and
the singularities at the general point of $\ell_3$ are of type \type{A_2};

\item
\label{thm:V4Sing-nG} 
the set of non-Gorenstein
points of $X$ consists of a point $P_3$ of type $\frac13(1,1,-1)$ and
two points $P_5$ and $P_5'$ of type $\frac15(1,2,-2)$, in particular, any non-Gorenstein singularity of $X$ is terminal.
\end{enumerate}
\end{theorem}

\begin{proof}
Let $X:=\PP(V_4')/\Alt_5$ and let
$\pi\colon \PP^3\to X$ be the quotient map.

\begin{sclaim}
\label{claim:V4prime}
$\PP(V_4')/\Alt_5$ is isomorphic to a hypersurface of degree $20$ in 
$\PP(2,3,4,5,10)$.
\end{sclaim}

\begin{proof}
We can regard $V_4'$ as a subrepresentation of
the five-dimensional permutation representation $V_{\mathrm{perm}}$ so that $V_4'\subset V_{\mathrm{perm}}$ is
defined by the linear equation $\sum x_i=0$, where $x_1,\dots, x_5$ are coordinates in $V_{\mathrm{perm}}=\CC^5$.
The ring of invariants 
$\CC[V_{\mathrm{perm}}]=\CC[x_1,\dots,x_5]^{\Alt_5}$ is generated by 
the elementary symmetric polynomials 
$z_1$,\dots, $z_5$ and the Vandermonde polynomial 
\begin{equation}
\label{eq:disc5}
z_{10}=\prod_{i>j} (x_i-x_j).
\end{equation}
There is a unique relation
\begin{equation}
\label{eq-eq} 
z_{10}^2=\Disc (z_1,\dots,z_5),
\end{equation}
where $\Disc$ is the discriminant, a
polynomial of weighted degree $20$ with respect to weights 
$\deg z_k=k$. Since $V_4'\subset V_{\mathrm{perm}}$, the functions $z_2$,\dots, $z_5$, and $z_{10}$
determine an embedding of $X$ 
into the weighted projective space $\PP(2,3,4,5,10)$
and the image is given by 
\begin{equation}
\label{eq-eq-1}
z_{10}^2=\Psi(z_2,\dots,z_5),\text{\quad where 
$\Psi(z_2,\dots,z_5):=\Disc(0,z_2,\dots,z_5)$}.\qedhere
\end{equation}
\end{proof}
Using Macaulay2 \cite{M2} we obtain the explicit form of $\Psi$:
\begin{equation}
\label{eq:Psi}
\begin{array}{lcl}
\Psi&=& 3125z_{5}^{4}-3750z_{2}z_{3}z_{5}^{3}
\\
&+&(2000z_{2}z_{4}^{2}+108z_{2}^{5}+
825z_{2}^{2}z_{3}^{2}-900z_{2}^{3}z_{4}+2250z_{3}^{2}z_{4})z_{5}^{2}
\\
&+&(16z_{2}^{3}z_{3}^{3}-
72z_{2}^{4}z_{3}z_{4}+108z_{3}^{5}-1600z_{3}z_{4}^{3}-630z_{2}z_{
3}^{3}z_{4}+560z_{2}^{2}z_{3}z_{4}^{2})z_{5}
\\ 
&+&16z_{2}^{4}z_{4}^{3}-4z_{2}^{3}z_{3}^{2}z_{4}^{2}
-27z_{3}^{4}z_{4}^{2}+144z_{2}z_{3}^{2}z_{4}^{3}-128z_{2}^{2}z_{4}^{4} 
+256z_{4}^{5}.
\end{array}\end{equation}
Note that $X$ has only quotient singularities. 
In particular, the singularities of $X$ are log terminal and $\QQ$-factorial \cite[Proposition~5.20, Lemma~5.16]{KM:book}.
The eigenspaces of any non-identity element of $\Alt_5$ have dimension $\le 2$.
Hence the morphism $\pi\colon \PP^3\to X$ is \'etale in codimension one
and so
\begin{equation}
\label{eq:KP3}
\pi^*K_X=K_{\PP^3}.
\end{equation} 
Therefore, $-K_X$ is ample and 
$(-K_X)^3=\frac1{60}(-K_{\PP^3})^3= 16/15$. 
This proves \ref{thm:V4prime1}.

\begin{sclaim}
\label{claim:torsion:free}
The group $\Cl(X)$ is torsion free.
\end{sclaim}

\begin{proof}
Assume that $\Cl(X)$ contains an element of order $n>1$.
It defines a cyclic degree $n$ cover $X'\to X$ that is \'etale outside $\Sing(X)$. 
Then we have the following commutative diagram of finite morphisms 
\[
\xymatrix@R=1em{
X''\ar[r]^{\pi'}\ar[d]_{\phi'}& X'\ar[d]_{\phi}
\\
\PP^3\ar[r]^{\pi} & X
} 
\]
where $X''$ is the normalization of the dominant component of $\PP^3\times_{X} X'$
and $\pi'$ is the quotient morphism by $\Alt_5$.
The morphism $\phi'$ is \'etale in codimension one, hence it splits, i.e. $X''$ is a disjoint union on $n$ copies of $\PP^3$. 
But then we have $X'= \sqcup_{i=1}^n X$, a contradiction.
\end{proof}

\noindent
\textit{Proof of Theorem~\xref{thm:V4prime} \textup(continued\textup).}
Now the assertion \ref{thm:V4prime2} follows from this claim and 
the fact $\Cl(X)\otimes \QQ=\Cl(\PP^3)^{\Alt_5}\otimes \QQ$.
Below, for short, we put $\PP:=\PP(2,3,4,5,10)$.
By Claim~\ref{claim:V4prime} we have an embedding $X\subset \PP$
as a hypersurface of degree $20$, hence by the adjunction formula 
we obtain
\begin{equation}
\label{eq-KX}
\OOO_{X}(K_X)\simeq\OOO_X(-2-3-4-5-10+20)\simeq\OOO_X(-4).
\end{equation}
Thus $K_X\sim -4A_X$, where $A_X=A_{\PP}|_X$ and
$A_{\PP}$ is the divisor corresponding to the reflexive sheaf $\OOO_{\PP}(1)$. 
On the other hand, by 
\eqref{eq:KP3}  the canonical class  $K_X$ is not divisible by a number greater than $4$.
This proves \ref{thm:V4prime3}. Note that there is a natural isomorphism 
\[
H^0\big(\PP,\, \OOO(n)\big) \simeq H^0\big(X,\, \OOO(nA)\big).
\]
Hence $\dim |-K_X|=1$ and a general member $S\in |-K_X|$
is given by $\mu z_4+\lambda z_2^2=0$. 
Then the divisor $\pi^*S$ is given in $\PP^3$ by the same equation 
$\mu z_4+\lambda z_2^2=0$, where $z_4$ and $z_2$ are regarded as polynomials in $x_1,\dots,x_5$.
Then it is easy to see that the singularities of $\pi^*S$ are Du Val.
In particular, the pair $(\PP^3,\pi^*S)$ is plt.
By \cite[Proposition~5.20]{KM:book} so is the pair $(X,S)$.
On the other hand, $K_X+S\sim 0$.
Hence, all 
the log discrepancies $a(\ \cdot\,,X,S)$ are integral and the pair
$(X,S)$ is canonical. 
This proves \ref{thm:V4prime4}.

One component of the locus of Gorenstein singularities of $X$ is contained in $X\cap \Sing(\PP)$, so we may assume that 
$\ell_2=X\cap \{z_3=z_5=0\}$. Other components are contained in $X\cap \{z_{10}=\partial \Psi/\partial z_i=0 \mid i=2,3,4,5\}$.
Using \eqref{eq:Psi} one can compute that there are two such components $\ell_2'$ and $\ell_3$.
This proves \ref{thm:V4Sing-G}.

If $P\in X$ is a non-Gorenstein point, then $P$ lies in $\Bs |-K_X|=\{z_2=z_4=0\}\cap X$.
On the other hand, $P$ lies also in 
the singular locus of $\PP$ which is given by $\{z_3=z_5=0\}\cup \{z_2=z_3=z_4=0\}\cup \{z_2=z_4=z_5=z_{10}=0\}$. 
Taking the equation \eqref{eq:Psi} into account, we get two possibilities:
$P=\{z_2=z_4=z_5=z_{10}=0\}$ and $P\in \{z_2=z_3=z_4= z_{10}^{2}-3\,125\,z_{5}^{4} =0\}$. 
This proves \ref{thm:V4Sing-nG}
\end{proof}

\begin{theorem}[{\cite{Kervaire-Vust}}, {\cite{Maeda-T-1989}}]
\label{thm:C4/A5}
The variety $\PP(V_4')/\Alt_5$ is rational.
\end{theorem}

\begin{proof}
In the notation of the proof of Claim~\ref{claim:V4prime},
let $P_1,\dots,P_5\in \PP(V_4')$ be the points corresponding to 
the images of the permutation basis in $V_{\mathrm{perm}}$ under the $\Alt_5$-equivariant projection
$V_{\mathrm{perm}}\to V_4'$. 
It is known that there exists the following $\Alt_5$-equivariant Sarkisov link (see {\cite[Proposition~4.6]{P:Inv10}} or 
{\cite[Remark~2.33]{Cheltsov-Kuznetsov-Shramov}})
\begin{equation}
\vcenter{
\xymatrix{
&Y\ar[dl]_{\sigma}\ar[dr]^{\phi_0}\ar@{-->}[rr]^{\chi}&&Y^+\ar[dl]_{\phi_0^+}\ar
[dr]^{\varphi}
\\
\PP^3&&\mathcal Y_3&&S
}}
\end{equation} 
where $\sigma$ is the blowup of $P_1,\dots,P_5$, \ $\chi$ is a flop, $\phi_0$ and $\phi_0^+$ are small contractions to the 
Segre cubic $\mathcal{Y}_3$ (see \cite[\S~9.4.4]{Dolgachev-ClassicalAlgGeom}), $S$ is a smooth del Pezzo surface of degree $5$, 
and $\varphi$ is a $\PP^1$-bundle. Note that this diagram is $\Sym_5$-equivariant as well.
By taking quotients we obtain an extended diagram
\begin{equation*}
\vcenter{
\xymatrix{
&Y\ar[d]_{\bar \pi}\ar[dl]_{\sigma}\ar[dr]^<(+.3){\phi_0}\ar@{-->}[rr]^{\chi}&&Y^+\ar[d]^{\pi^+}\ar[dl]_<(+.3){
\phi_0 ^+}\ar
[dr]^{\varphi}
\\ 
\phantom{a}\PP^3\ar[d]_{\pi}&Y/\Alt_5\ar[dr]^{\bar\phi_0}\ar[dl]_{\bar\sigma}
\ar@{-->}@/^2em/[rr]^{\bar\chi}
&\mathcal{Y}_3\ar[d]&Y^+/\Alt_5\ar[dr]^{\bar\varphi}\ar[dl]_{\bar\phi_0^+} &S\ar[d]
\\
X=\PP^3/{\Alt_5}&&\mathcal{Y}_3/\Alt_5&&S/{\Alt_5}
}}
\end{equation*}
The quotient $S/\Sym_5$ is studied in detail in Sect.~\ref{sect:DP5}, but here we do not need 
that explicit description.

The action of $\Alt_5$ on $Y$ is free in codimension one, hence 
similar to Claim~\ref{claim:torsion:free}
one can show that the group $\Cl(Y/\Alt_5)$ is torsion free.
By the construction 
\[
K_Y\sim \sigma^*\pi^* K_X+2\sum E_i\sim \bar\pi^*\bar\sigma^*(-4A_X)+2\bar\pi^* F,
\]
where the $E_i$'s are $\sigma$-exceptional divisors and $F=\bar\pi(E_i)$ is 
the $\bar\sigma$-exceptional divisor.
Therefore, the canonical class $K_{Y/\Alt_5}$ is divisible by $2$ in the group $\Cl(Y/\Alt_5)$. 
Since the map $\bar\chi$ is an isomorphism in codimension 1,
the divisor $K_{Y^+/\Alt_5}$ is also divisible by $2$.
The general fiber of $\bar\varphi$ is a smooth rational curve, i.e. $Y^+/\Alt_5$ is a Severi-Brauer scheme over 
an open subset in $S/\Alt_5$. The divisor $-\frac 12 K_{K_{Y^+/\Alt_5}}$ provides its birational section,
hence the variety $Y^+/\Alt_5$ is birational to $\PP^1\times S/\Alt_5$,
hence it is rational.
\end{proof}

\section{Quotients $(\PP^1\times \PP^1)/\Alt_5$}
\label{Sect:P1P1}
\subsection{}
\label{subSect:P1P1}

In this section we describe quotients of $\PP^1\times \PP^1$ by the icosahedral group.
Note that all non-trivial actions of $\Alt_5$ on rational ruled surfaces are $\Alt_5$-birationally equivalent \cite[Theorem~B.10]{Cheltsov:2ineq}.
We restrict ourselves to actions on $\PP^1\times \PP^1$ that are non-trivial on each factor.
Then, up to outer automorphisms there are two choices:
\begin{enumerate}
\renewcommand\labelenumi{\rm (\alph{enumi})}
\renewcommand\theenumi{\rm (\alph{enumi})}
\item
\label{rem:Cr2:diag}
$\Alt_5 \acts \PP(V_2)\times \PP(V_2)$ and
\item
\label{rem:Cr2:tdiag}
$\Alt_5 \acts \PP(V_2)\times \PP(V_2')$.
\end{enumerate}
We call them \textit{diagonal} and \textit{twisted diagonal}, respectively.
If we fix any embedding $\Alt_5 \subset \Aut(\PP^1)=\PGL_2(\CC)$, then 
we may assume that the action $\Alt_5$ on $\PP^1\times \PP^1$ is given by
\[
s (P_1,P_2)= \begin{cases}
\Big(s(P_1), s(P_2)\Big) &\text{in the case \ref{rem:Cr2:diag},}
\\
\Big(s(P_1), \alpha(s)(P_2)\Big)&\text{in the case \ref{rem:Cr2:tdiag},}
\end{cases}
\]
where $\alpha$ is an outer automorphism of $\Alt_5$.

\begin{sremark}
\label{rem:autP1P1}
Recall that the group $\Aut(\PP^1\times \PP^1)$ can be presented as the semi-direct product
$\big(\Aut(\PP^1)\times \Aut(\PP^1)\big) \rtimes \langle \tau\rangle$, where
$\tau$ is the involution switching the factors.
It is easy to see that $\tau$ normalizes the image of $\Alt_5$ in $\Aut(\PP^1\times \PP^1)$ and so
the normalizer $G$ of $\Alt_5\subset \Aut(\PP^1\times \PP^1)$ is the group of order $120$.
On the other hand, we have a natural homomorphism $\varphi:G\to \Aut(\Alt_5)\simeq \Sym_5$ whose
restriction to 
$\Alt_5$ is injective. Hence either $\varphi$ is an isomorphism or its kernel is of 
order $2$ and in the latter case $G=\Alt_5\times \ker(\varphi)$.
Therefore, $G$ is isomorphic to $\Alt_5\times \Sym_2$ in the case \ref{rem:Cr2:diag} and to $\Sym_5$ in the case \ref{rem:Cr2:tdiag}.
\end {sremark}

\subsection{The diagonal action}

\begin{stheorem}
\label{thm:P1P1:diag}
Let $\Alt_5$ act on $\PP^1\times \PP^1$ diagonally and let $G$ be 
the normalizer of $\Alt_5$ in $\Aut(\PP^1\times \PP^1)$. Then
$G=\Alt_5\times 
\langle\tau\rangle$ and the following assertions hold. 
\begin{enumerate}
\item 
\label{thm:P1P1:diag:a}
The quotient $(\PP^1\times \PP^1)/G$ is the weighted projective 
plane $\PP(1,6,10)\simeq \PP(1,3,5)$.

\item 
\label{thm:P1P1:diag:b}
The quotient $(\PP^1\times \PP^1)/\Alt_5$ is the hypersurface given by the 
equation $z_{15}^2=\Phi(z_1^2, z_6, z_{10})$ in $\PP(1, 6,10,15)$, where
\begin{equation}
\label{eq:Phi-z2}
\Phi(z_1^2, z_6, z_{10})=-1728z_6^5+z_{10}^3+720 z_1^2z_6^3z_{10}-80 z_1^4 z_6z_{10}^2+64z_1^6 (5z_6^2- z_1^2z_{10})^2.
\end{equation}

\item 
\label{thm:P1P1:diag:c}
The morphism $\PP^1\times \PP^1\to (\PP^1\times \PP^1)/\Alt_5$ is \'etale in codimension one.

\item
\label{thm:P1P1:diag:d}
The natural double cover $(\PP^1\times \PP^1)/\Alt_5\to (\PP^1\times \PP^1)/G$
is the projection along $z_{15}$-axis and it is branched over the curve $D$ given by
$\Phi(z_1^2, z_6, z_{10})=0$. In particular, $D$ does not pass through singular 
points of $(\PP^1\times \PP^1)/G\simeq \PP(1,3,5)$.
\end{enumerate}
\end{stheorem}

\begin{proof}
We identify $\PP^1\times \PP^1$ with $\PP(V_2)\times \PP(V_2)$ and
consider the natural Segre embedding $\PP(V_2)\times \PP(V_2)\hookrightarrow \PP(V_2\otimes V_2)$
whose image is a smooth quadric.
Let $u_0,u_1$ (resp. $v_0,v_1$) be coordinates in the first (resp. the second) copy of $V_2$.
We may assume that the bilinear map $V_2\oplus V_2 \longrightarrow V_2\otimes V_2$ is given by
\[
(u_0,\, u_1;\, v_0,\, v_1) \longmapsto (u_0v_0,\, u_0v_1,\, u_1v_0,\, u_1v_1) 
\]
and then, in the corresponding coordinates $y_1,\dots, y_4$, the equation of $\PP^1\times \PP^1$ can be written as follows:
\[
y_1y_4-y_2y_3=0.
\]
The involution 
$\tau$ is induced by 
\[
\tau: V_2\oplus V_2 \longrightarrow V_2\oplus V_2,\qquad (u_0,\, u_1;\, v_0,\, v_1) \longmapsto (-v_0,\, -v_1;\, u_0,\, u_1).
\]
Note that
$V_2\otimes V_2=\Sy^2 (V_2) \oplus \wedge^2(V_2)=V_3\oplus V_1$ (see Table~\ref{tab:repSL25}).
We have $V_3=\{y_2=y_3\}$ and $V_1=\{y_1=y_4=y_2+y_3=0\}$.
Now let $x_0=y_1$, $x_1=y_2+y_3$, $x_2=y_4$, $z_1=y_2-y_3$, i.e. $x_0,x_1,x_2$ are coordinates in $V_3$ and $z_1$ is a coordinate in 
$V_1$. Then the involution $\tau$ acts on $V_3\oplus V_1$ via $(x_0,x_1,x_2,z_1) 
\mapsto (-x_0,-x_1,-x_2,z_1)$.
This action commutes with that of $\Alt_5$. Hence $\N_{\Aut(\PP^1\times \PP^1)}(\Alt_5)=\Alt_5\times 
\langle\tau\rangle$.
Further, we have
\[
\CC[V_2\otimes V_2]^{\Alt_5}=\CC[V_3]^{\Alt_5}\otimes 
\CC[V_1]=\CC[x_0,x_1,x_2]^{\Alt_5}[z_1]=\CC[z_1,z_2, z_6, z_{10}, 
z_{15}]/(z_{15}^2-\Phi),
\]
where $z_2, z_6, z_{10}, z_{15}$ are invariant polynomials in $x_0,x_1,x_2$
considered in Sect.~\ref{sect:P2} and $\Phi$ is given by~\eqref{eq:Phi}. Similarly, 
\[
\CC[V_2\otimes V_2]^{G}=\CC[V_3]^{G}\otimes 
\CC[V_1]=\CC[x_0,x_1,x_2]^{G}[z_1]=\CC[z_1,z_2, z_6, z_{10}].
\]
Thus $\PP(V_2\otimes V_2)/G\simeq \PP(1,2,6,10)$ and
$\PP(V_2\otimes V_2)/\Alt_5$ is the hypersurface given by $z_{15}^2=\Phi$ in 
$\PP(1,2,6,10,15)$.
Up to scaling we may assume that the equation of $\PP(V_2)\times 
\PP(V_2)\subset 
\PP(V_2\otimes V_2)$ is $z_2-z_1^2=0$.
Then as in the proof of Proposition~\ref{prop:V3} we have
$(\PP^1\times \PP^1)/G\simeq \PP(1,6,10)\simeq \PP(1,3,5)$ and
$(\PP^1\times \PP^1)/\Alt_5$ is given by 
$z_{15}^2=\Phi(z_1^2, z_6, z_{10})$ in $\PP(1,6,10,15)$.
This proves~\ref{thm:P1P1:diag:a} and ~\ref{thm:P1P1:diag:b}.
The assertion~\ref{thm:P1P1:diag:c} follows from the fact that 
the diagonal action of $\Alt_5$ on $\PP^1\times \PP^1$ is free in codimension one
and~\ref{thm:P1P1:diag:d} follows from~\ref{thm:P1P1:diag:a} and~\ref{thm:P1P1:diag:b}.
\end{proof}

\begin{sremark}
Note that $\PP^1\times \PP^1/\langle \tau\rangle\simeq \PP^2$. Hence for the computation of 
$\PP^1\times \PP^1/G$ we can use Proposition~\ref{prop:V3}.
\end{sremark}

\subsection{The twisted diagonal action}

\begin{stheorem}
\label{thm:P1P1:skdia}
Let $\Alt_5$ act on $\PP^1\times \PP^1$ twisted diagonally and let $G$ be 
the normalizer of $\Alt_5$ in $\Aut(\PP^1\times \PP^1)$. Then $G\simeq 
\Sym_5$ and the following assertions hold. 
\begin{enumerate}
\item 
\label{thm:P1P1:skdia1}
The quotient $(\PP^1\times \PP^1)/G$ is the weighted projective 
plane $\PP(3,4,5)$.
\item 
\label{thm:P1P1:skdia2}
The quotient $(\PP^1\times \PP^1)/\Alt_5$ is the hypersurface given by the 
equation
\begin{equation}
\label{eq:skew-diagonal}
z_{10}^2=\Upsilon(z_3,z_4,z_5)\quad \text{in $\PP(3,4,5,10)$,}
\end{equation} 
where $\Upsilon=\Disc(0,0,z_3,z_4,z_5)$ \textup(see~\eqref{eq:Psi:00}\textup).

\item 
\label{thm:P1P1:skdia3}
The morphism $\PP^1\times \PP^1\to (\PP^1\times \PP^1)/\Alt_5$ is \'etale in 
codimension one. 

\item
\label{thm:P1P1:skdia4}
The natural double cover $\pi: (\PP^1\times \PP^1)/\Alt_5\to (\PP^1\times \PP^1)/G$
is the projection along $z_{10}$-axis and it is branched over the curve 
$D\subset \PP(3,4,5)$ given by $\Upsilon(z_3,z_4,z_5)=0$ \textup(see~\eqref{eq:Psi:00}\textup).
\item
\label{thm:P1P1:skdia5}
The curve $D$ has two singularities: an ordinary double point and a simple cusp; these singularities lie in 
the smooth locus of $\PP(3,4,5)$. 
\item
\label{thm:P1P1:skdia6}
The singularities of $(\PP^1\times \PP^1)/\Alt_5$ are as follows: 
two Du Val points of types \type{A_1} and \type{A_2} at $\pi^{-1}(\Sing(D))$,
a cyclic quotient singularity of type $\frac13(1,1)$ at $\pi^{-1}((1,0,0))$,
two Du Val singularities of type \type{A_3} at $\pi^{-1}((0,1,0))$, and two 
singularities of type
$\frac15(3,4)$ at $\pi^{-1}((0,0,1))$. 
\end{enumerate}
\end{stheorem}

\begin{proof}
We identify $\PP^1\times \PP^1$ with $\PP(V_2)\times \PP(V_2')$ and
consider the natural Segre embedding $\PP(V_1)\times \PP(V_2')\hookrightarrow \PP(V_2\otimes V_2')$
whose image is a smooth quadric. Note that $V_2\otimes V_2'\simeq V_4'$ (see Table~\ref{tab:repSL25}). As in
the proof of Theorem~\ref{thm:V4prime}, we can regard $V_4'$ as a subrepresentation of
the natural permutation representation $V_{\mathrm{perm}}$
so that $V_4'=\{\sum x_i=0\}$, where $x_1,\dots, x_5$ are coordinates in $V_{\mathrm{perm}}=\CC^5$.
Let $z_1,\dots, z_5$ be the elementary symmetric polynomials in $x_1,\dots, x_5$.
Then there are embeddings
\[
\iota: \PP(V_1)\times \PP(V_2') \xhookrightarrow{\hspace{0.7em} \iota_1\hspace{0.7em}} 
\PP(V_2\otimes V_2') \xhookrightarrow{\hspace{0.7em} \iota_2\hspace{0.7em}} 
\PP(V_{\mathrm{perm}}),
\]
where the image $\iota(\PP(V_1)\times \PP(V_2'))$ is given by $z_1=z_2=0$. 
Clearly, $\iota(\PP(V_1)\times \PP(V_2'))$ is invariant under the action of the symmetric group $\Sym_5$.
Therefore, $\N_{\Aut(\PP^1\times \PP^1)}(\Alt_5)=\Sym_5$
(see Remark~\ref{rem:autP1P1}).

Since
$\CC[V_{\mathrm{perm}}]^{\Sym_5}=\CC[z_1,\dots, z_5]$, we have an embedding
\[
\PP(V_1)\times \PP(V_2')/\Sym_5 \hookrightarrow
\PP(V_{\mathrm{perm}})/\Sym_5=\PP(1,2,3,4,5)
\]
whose image is given by $z_1=z_2=0$, where $z_1,\dots, z_5$ are regarded as quasihomogeneous coordinates 
in $\PP(1,2,3,4,5)$. Thus $(\PP^1\times \PP^1)/\Sym_5\simeq \PP(3,4,5)$.
This proves ~\ref{thm:P1P1:skdia1}.

Further, the ring of invariants
$\CC[x_1,\dots,x_5]^{\Alt_5}$ is generated by
$z_1$,\dots, $z_5$ and the discriminant $z_{10}$ (see \eqref{eq:disc5})
and there is a unique relation \eqref{eq-eq}.
Thus we have an embedding 
\[
(\PP^1\times \PP^1)/\Alt_5 \hookrightarrow \PP(1,2,3,4,5,10)
\]
so that the image is given by three equations $z_1=z_2=z_{10}^2-\Disc (z_1,\dots,z_5)=0$.
Hence $(\PP^1\times \PP^1)/\Alt_5$ is given in $\PP(3,4,5,10)$ by the equation
\eqref{eq:skew-diagonal}
\begin{equation*}
z_{10}^2=\Upsilon(z_3,z_4,z_5),\quad 
\text{where $\Upsilon=\Disc(0,0,z_3,z_4,z_5)=\Phi(0,z_3,z_4,z_5)$},
\end{equation*}
see Sect.~\ref{sect:V4prime}. From \eqref{eq:Psi} we obtain
\begin{equation}
\label{eq:Psi:00}
\Upsilon(z_3,z_4,z_5)=3\,125\,z_{5}^{4}+256\,z_{4}^{5} 
-27\,z_{3}^{4}z_{4}^{2}+108\,z_{3}^{5}z_ {5}- 
1\,600\,z_{3}z_{4}^{3}z_{5}+2\,250\,z_{3}^{2}z_{4}z_{5}^{2}.
\end{equation} 
This proves~\ref{thm:P1P1:skdia2}.

The assertion \ref{thm:P1P1:skdia3} follows from the fact that the action of $\Alt_5$ 
on $\PP(V_2)\times \PP(V_2')$ is free in codimension $1$ and 
\ref{thm:P1P1:skdia4} follows from \ref{thm:P1P1:skdia1} and \ref{thm:P1P1:skdia2}.

It is easy to see from \eqref{eq:Psi:00} that $D$ does not pass through 
points $(0:1:0)$ and $(0:0:1)$ and is smooth along the line $z_3=0$.
Next we compute the singularities of $D$ on the affine chart $U_3:=\{z_3\neq 0\}=\CC^2_{z_4,z_5}/\mumu_3(1,2)$.
The equation of $D$ in this chart has the form:
\begin{equation}
\label{eq:Psi:00a}
\Upsilon=
256\,z_{4}^{5}-1\,600\,z_{4}^{3}z_{5}+3\,125\,z_{5}^{4}+2\,250\,z_{4}z_{
5}^{2}-27\,z_{4}^{2}+108\,z_{5}.
\end{equation}
Near the origin, the pair $(\PP(3,4,5)\supset D)$ is locally isomorphic to 
$(\CC^2_{z_4,z_5}\supset\{z_5=0\})/\mumu_3(1,2)$, hence it is plt.
Finally, easy but lengthy computations show that in the chart $U_3$ 
the curve $D$ has two singular points: 
an ordinary double point at $\left(-\frac{3}{20}\sqrt[3]{25},\, \frac{9}{50}\sqrt[3]{5}\right)$
and a simple cusp at
$\left(\frac{3}{20} \sqrt[3]{100},\, \frac{3}{50} \sqrt[3]{10} \right)$.
This proves \ref{thm:P1P1:skdia5}.
Finally, \ref{thm:P1P1:skdia6} follows from \ref{thm:P1P1:skdia5} and \ref{thm:P1P1:skdia2}.
\end{proof}

\subsection{Rationality of the quotients of $\PP^3$ by imprimitive groups}
As application of the above results we prove the rationality of quotients of $\PP^3$ by some imprimitive groups.
A group $G\subset \GL(V)$ of linear transformations is said to be \textit{imprimitive} if the space $V$
admits
a decomposition 
\begin{equation}
\label{eq:def:imp}
V=\bigoplus_{i=1}^m V^{(i)}
\end{equation} 
such that $G$ permutes the $V^{(i)}$'s.
In other words, the representation $G\hookrightarrow \GL(V)$
is induced by a lower-dimensional representation of a proper subgroup $H\subset G$.

Is a consequence of the computations above, 
we prove the following result.
\begin{stheorem}
\label{thm:imp}
Let $G\subset \GL(V)$ be a finite imprimitive group, where $\dim V=4$.
Suppose that $G$ is not solvable. Then the variety $\PP(V)/G$ is rational.
\end{stheorem}
Note that more general result is announced in \cite{P:Inv10}. The proof in that case follows the same ideas but is
much longer.

\begin{proof}
We may assume that the representation $G\hookrightarrow \GL(V)$ is irreducible.
Then the decomposition \eqref{eq:def:imp} induces a homomorphism $\nu: G\to \Sym_m$ 
whose image is transitive.
Let $N\subset G$ be the kernel of $\nu$.
If $m=4$, then $N$ is abelian and $G$ must be solvable. This contradicts our assumption.
Thus $m=2$ and $\dim V^{(i)}=2$. 
We put $V':=V^{(1)}$ and $V'':=V^{(2)}$.
Then $\PP(V')$ and $\PP(V'')$ are skew lines in $\PP(V)$
that are $N$-invariant and switched by $G/ N$.
There exists the following $G$-equivariant Sarkisov link 
\[
\vcenter{
\xymatrix{
&\widetilde {\PP(V)}\ar[dl]_{\sigma}\ar[dr]^{\varphi}&
\\
\PP(V)\ar@{-->}[rr]&& \PP(V')\times \PP(V'')
} }
\]
where $\sigma$ is the blowup of $\PP(V')\cup \PP(V'')$ and 
$\varphi$ is a $\PP^1$-bundle (see e.g.~\cite[Sect.~6]{P:Inv10}).

Let $\bar G$ and $\bar N$ be images in $\Aut(\PP^1\times\PP^1)$ of $G$ and $N$. respectively.
Then $\bar N\subset \Aut(\PP^1)\times \Aut(\PP^1)$. Let $p_i:\bar N \to \Aut(\PP^1)$ is the projection to the $i$-th factor
and let $\bar N_i:= p_i(\bar N)\subset \Aut(\PP^1)$. Then $\bar N\subset\bar N_1\times \bar N_2$.
Since $\bar N$ is not solvable, the only possibility is $\bar 
N_1\simeq \bar N_2\simeq \Alt_5$. 
Let $K_1:=\ker (p_2)\cap\bar N_1$ and $K_2:=\ker (p_1)\cap\bar N_2$. By Goursat's lemma (see e.g. \cite[Ch.~1, Ex.~5]{Lang:Algebra}) we have
$\bar N_1/K_1\simeq \bar N_2/K_2$ and $\bar N$ is isomorphic to the fiber product $\bar N_1\times_{\bar N_i/K_i} \bar N_2$. 

Consider the case $K_i=\bar N_i$, i.e. $\bar N=\bar N_1\times\bar N_2\simeq 
\Alt_5\times \Alt_5$.
Then $(\PP^1\times \PP^1)/\bar N = \PP^1/\bar N_1\times \PP^1/\bar N_2 \simeq \PP^1\times \PP^1$. Moreover, $\bar G/\bar N$ acts on this 
surface interchanging the factors. Therefore, 
\[
(\PP^1\times \PP^1)/ \bar G = (\PP^1\times \PP^1)/(\bar G/\bar N)\simeq \PP^2.
\]
The branch divisor of the quotient morphism
\[
\theta: \PP^1\times \PP^1 \to \PP^1/\bar N_1\times \PP^1/\bar N_2= \PP^1\times \PP^1
\]
consists of three curves $L_1,L_2,L_3\subset \PP^1\times \PP^1$ of bidegree $(1,0)$ and 
three curves $L_1',L_2',L_3'\subset \PP^1\times \PP^1$ of bidegree $(0,1)$.
The branch divisor of the double cover $\pi: \PP^1\times \PP^1 \to (\PP^1\times \PP^1)/(G/N)\simeq \PP^2$
is a conic $C\subset \PP^2$ which is the image of the diagonal $\Theta\subset \PP^1\times \PP^1$.
With appropriate numbering, $\Theta$ passes through intersection points $P_1:=L_1\cap L_1'$, $P_2:=L_2\cap L_2'$, $P_3:=L_3\cap L_3'$.
Thus the branch divisor of the quotient map $\pi \comp \theta:\PP^1\times \PP^1 \to (\PP^1\times \PP^1)/N= \PP^2$
consists of a conic $C$ and three lines $M_i:=\pi(L_i)=\pi( L_i')$ touching $C$ at $Q_i:=\pi (P_i)$
(see Fig.~\ref{fig:ell}). Then the rationality of $\PP^3/G$ follows from Lemma~\ref{lemma:P2}.

Now, assume that $K_i=\{1\}$, i.e. $\bar N\simeq\bar N_1\simeq \bar N_2\simeq 
\Alt_5$. 
Then we have two possibilities \ref{rem:Cr2:diag} and~\ref{rem:Cr2:tdiag} of 
Subsection~\ref{subSect:P1P1}.
If $\bar N=\Alt_5$ acts on $\PP^1\times \PP^1$ diagonally, then $\bar G\simeq 
\Alt_5\times \langle\tau\rangle$, and $(\PP^1\times \PP^1)/\bar G\simeq 
\PP(1,3,5)$
(see Theorem~\ref{thm:P1P1:diag}). The curve $D$ does not pass 
through singular points of $S$.
Then the rationality of $\PP^3/G$ follows from Lemma \ref{lemma:P(1,3,5)}.

Finally, assume that $\Alt_5$ act on $\PP^1\times \PP^1$ twisted diagonally. 
Then $\bar G\simeq
\Sym_5$, and $(\PP^1\times \PP^1)/\bar G\simeq \PP(3,4,5)$ (see Theorem~\ref{thm:P1P1:skdia}).
Then the rationality of $\PP^3/G$ follows from Lemma \ref{lemma:P(3,4,5)}.
\end{proof}

\subsection{}
As a byproduct of the results in Theorems~\ref{thm:P1P1:diag} and~\ref{thm:P1P1:skdia} we also prove the rationality of $\PP^3/\Sym_5$, where the action $\Sym_5\acts \PP^3$ is 
induced by a faithful representation  of a central extension  of~$\Sym_5$:

\begin{stheorem}[cf. {\cite[Theorem~1.4]{Plans:Noether}}]
\label{th-tildeS5}
Let $\tilde \Sym_5$ be a central extension of $\Sym_5$ and let $V$ be its four-dimensional representation. Then
$\PP(V)/\tilde \Sym_5$ is rational. 
\end{stheorem}
\begin{proof}
We may assume that  $V$ is irreducible and faithful.
Let $H:=[\tilde \Sym_5, \tilde \Sym_5]$ be the derived subgroup and let $\z(\tilde \Sym_5)$ be the center of $\tilde \Sym_5$.
Then $\z(\tilde \Sym_5)$ is cyclic and  $H/(\z(\tilde \Sym_5)\cap H)\simeq \Alt_5$.
Assume that $H\simeq \Alt_5$. Then as a $H$-module $V$ is isomorphic to $V_4'$.
Hence there exists  $H$-invariant irreducible  surfaces   $Q,\, S\subset \PP(V)$ of degree $2$ and $3$, respectively
(see the proof of Claim~\ref{claim:V4prime}). The intersection $\Gamma:=Q\cap S$ is a smooth
canonical curve of genus $4$. It is easy to see that $\Gamma$ is $\tilde \Sym_5$-invariant, so 
$\Sym_5=\tilde \Sym_5/ \z(\tilde \Sym_5)$ faithfully acts on $\Gamma$. Since $\Gamma$ is a canonical curve,
we have a natural identification  $V\simeq H^0(\Gamma, \OOO_\Gamma(K_\Gamma)^\vee$, hence 
$\tilde \Sym_5$ contains $\Sym_5$ as a subgroup. Then we may assume that $\tilde \Sym_5=\Sym_5$
and $V$ as a  $\Sym_5$-module is isomorphic to either $W_4$ or $W_4'$ (see Table~\ref{Tab:repS5}).
The $\Sym_5$-varieties $\PP(W_4)$ and $\PP(W_4')$ are isomorphic and $\PP(W_4)\simeq \PP(2,3,4,5)$ is rational
(see the proof of Claim~\ref{claim:V4prime}). This proves our theorem in the case $H\simeq \Alt_5$.

Thus we may assume that $H\simeq\tilde  \Alt_5$.
If the representation of $\tilde \Alt_5$ on $V$ is reducible, then $\tilde \Sym_5\subset \GL(V)$ is
imprimitive. Then $\PP(V)/\tilde \Sym_5$ is rational by Theorem~\ref{thm:imp}.
Thus we may assume that $V$ is irreducible as a $\tilde \Alt_5$-module,
so $\PP(V)$ as an $\tilde \Alt_5$-variety is isomorphic to $\PP(V_4)$. 
Since $V_4\simeq \Sy^3(V_2)$ (see Table~\ref{tab:repSL25}),
there exists
an invariant twisted cubic curve $C_1=\upsilon_3(\PP^1)\subset \PP(V_4)$.
Its image under the action by elements $\tilde \Sym_5\setminus\tilde \Alt_5$ is another 
twisted cubic curve $C_2\subset \PP^3$. Then the pair $C_1\cup C_2$ is $\tilde \Sym_5$-invariant.
In fact, the decomposition
$\Sy^2 (V_4^\vee)=V_3\oplus V_3'\oplus V_4'$ of the symmetric square defines $C_1$ and $C_2$:
up to permutations we may assume that the (quadratic) equations of $C_1$ sit in $V_3$ while the equations of $C_2$ sit in $V_3'$.
The intersection $C_1\cap C_2$ consists of at most $6$ points which cannot be permuted by $\tilde \Alt_5$.
This implies that $C_1\cap C_2=\varnothing$.
Now according to \cite[Proposition~2.9]{Cheltsov-Shramov:P3} there exists the following $\tilde \Sym_5$-equivariant 
Sarkisov link:
\[
\vcenter{
\xymatrix{
& \tilde \PP^3\ar[dl]_{\sigma}\ar@{-->}[rr]^{\chi}&&Y\ar[dr]^{\varphi}
\\
\PP^3 &&&&\PP^1\times \PP^1
} }
\]
where $\sigma$ is the 
blowup of $C_1\cup C_2$, $\chi$ is a flop, $\varphi$
is a $\PP^1$-bundle, and the induced action of $\Sym_5$ on $\PP^1\times\PP^1$ is faithful.
Then $(\PP^1\times \PP^1)/\Sym_5\simeq \PP(3,4,5)$ by Theorem~\ref{thm:P1P1:skdia} and $Y/\Sym_5$ is rational by Lemma~\ref{lemma:P(3,4,5)}.
\end{proof}

\section{Quintic del Pezzo surface}
\label{sect:DP5}
Throughout this section $S$ denotes the quintic del Pezzo surface.
Recall that its automorphism group is isomorphic to the symmetric group $\Sym_5$
(see e.g.~\cite[Theorem~8.5.8]{Dolgachev-ClassicalAlgGeom}). 

\begin{theorem}
\label{thm:DP5}
Let $S$ be the quintic del Pezzo surface. Then the following assertions hold.
\begin{enumerate}
\item \label{thm:DP5a}
$S/\Aut(S)\simeq\PP(1,2,3)$.
\item \label{thm:DP5b}
$S/\Alt_5\simeq \PP(1,1,3)$.\footnote{This fact was proved also in \cite{Trepalin:Quot-high}.}
\item\label{thm:DP5c}
The natural double cover $S/\Alt_5\to S/\Aut(S)$ is branched over an irreducible conic $\bar\DLines\subset S/\Aut(S)=\PP(1,2,3)$.
\end{enumerate}
\end{theorem}

The whole section will be devoted to the proof of Theorem~\ref{thm:DP5}.
In contrast with the cases of quotients of $\PP^2$ and $\PP^1\times \PP^1$
out proof here is indirect.
We need a lot of preliminary technical results on the action 
of $\Aut(S)$. Most of them can be deduced from the corresponding results 
presented in \cite[Sect.~6.2]{CheltsovShramov:book} and \cite[Sect.~3]{Dolgachev-Farb-Looijenga}.
We provide independent proofs.

\subsection{Standard facts and notation (see e.g.~\cite[\S~8.5]{Dolgachev-ClassicalAlgGeom}).}
The anticanonical linear system $|-K_S|$ on the quintic del Pezzo surface is very ample and defines
an embedding
$S\subset \PP^5$ so that the image is an intersection of quadrics. This
embedding will be fixed throughout this section.
There is an isomorphism $\Aut(S)\simeq \Sym_5$ which will also be fixed.  The
action $\Aut (S) \acts S$
is induced by projective transformations, which, in turn, is induced 
by the representation $H^0(S, \OOO_S(-K_S))\simeq W_6$ \cite{Shepherd-Barron-S5}.
The surface $S$ contains exactly $10$ lines.
Denote by $\Lines$ the set of lines on $S$ and by $\DLines$ we 
denote the divisor $\sum_{L\in \Lines} L$.
We have 
\begin{equation}
\label{eq:DLines}
\DLines\sim -2K_S.
\end{equation} 
There are at most two lines passing through any point of $S$.
Let $\Points$ be the set of intersection points $L_1\cap L_2$, where $L_1,\, 
L_2\in \Lines$, $L_1\neq L_2$. This set consists of $15$ elements.
There is an equivalence relation on $\Points$: we say that 
$P\approx P'$ for $P,\, P'\in \Points$ if the divisors $L_1+L_2$ and $L_1'+L_2'$
are linearly equivalent, where $L_1,L_2$ (resp. $L_1',L_2'$) are lines passing trough $P$ (resp. 
$P'$).
The incidence graph of the set of lines on $S$ is the famous Petersen
graph (see Fig.~\ref{Petersen:graph}).
\begin{figure}[H]
\begin{tikzpicture}[
xscale=0.4, 
yscale=0.4,
vertex_style/.style={draw, fill, circle, inner sep=1pt},
edge_style/.style={ black,drop shadow={opacity=0}}
]
\useasboundingbox (-5.05,-4.4) rectangle (5.1,5.25);

\begin{scope}[rotate=90]
\foreach \x/\y in {0/1,72/2,144/3,216/4,288/5}{
\node[vertex_style] (\y) at (canvas polar cs: radius=2.5cm,angle=\x){};
}
\foreach \x/\y in {0/6,72/7,144/8,216/9,288/10}{
\node[vertex_style] (\y) at (canvas polar cs: radius=5cm,angle=\x){};
}
\end{scope}
\foreach \x/\y in {1/6,2/7,3/8,4/9,5/10}{ \draw[edge_style] (\x) -- (\y);}

\foreach \x/\y in {1/3,2/4,3/5,4/1,5/2}{
\draw[edge_style] (\x) -- (\y);
}

\foreach \x/\y in {6/7,7/8,8/9,9/10,10/6}{
\draw[edge_style] (\x) -- (\y);
}
\end{tikzpicture}
\caption{Petersen graph}
\label{Petersen:graph}
\end{figure}
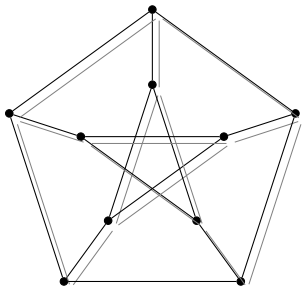
\begin{sremark}
\label{rem:S5-grA5}
As an $\Alt_5$-module, the space $H^0(S, \OOO_S(-K_S))$ has the following decomposition: 
\[
H^0(S, \OOO_S(-K_S))\simeq V_3\oplus V_3'.
\]
This defines an $\Sym_5$-equivariant morphism $S \to \PP(V_3)\times \PP(V_3')$.
The morphism is birational onto its image but not an isomorphism. 
It is described in \cite[Sect.~4]{Dolgachev-Farb-Looijenga}.
\end{sremark}

\subsection{Conic bundles}
\label{conic:bundles}
For any conic $C\subset S$, we have $\dim |C|=1$ and the pencil $|C|$ defines a 
morphism $\varphi: S\to \PP^1$, which is a conic bundle having exactly three 
degenerate fibers, each of them is a pair of lines. 
The linear system 
$|-K_S-C|$ defines a birational contraction $\psi: S\to \PP^2$ of four 
disjoint lines that are sections of $\varphi$. 
Thus any line on $S$ is either a component of degenerate fiber of $\varphi$ or 
an exceptional curve of $\psi$. 
For any point $P\in \Points$ there are exactly two other points $P_1,\, P_2\in \Points$
such that $P\approx P_1 \approx P_2$.
The morphisms $\varphi$ and $\psi$ are 
equivariant with respect to a subgroup $\Sym_4\subset \Sym_5$, so there 
are exactly five conic bundle structures on $S$, they correspond to 
equivalence classes $\Points/{\approx}$, and the group $\Sym_5$ 
acts on the set of conic bundles on $S$ transitively.
There are exactly $5$ (smooth) conics passing through any point $P\in 
S\setminus \DLines$.
For any two conics $C_1$ and $C_2$ on $S$ we have $C_1\cdot C_2\le 1$. Moreover, if $C_1\cdot 
C_2=0$, then $C_1\sim C_2$. In particular, this implies that there are no 
smooth conics passing through any point $P\in \Points$.

\begin{snotation}
For $P\in \Points$ we denote by $C_P$ the degenerate conic $L_1+L_2$, where $L_1$ and $L_2$ are
lines passing through 
$P$. The 
corresponding conic bundle given by $|C_P|$ we denote by $\varphi_P: S\to \PP^1$.
Recall that for any point $P\in \Points$ there exists exactly two 
lines passing through~$P$.
\end{snotation}

\subsection{Twisted cubic curves}
If $R\subset S\subset \PP^5$ is a twisted cubic curve, then 
by the adjunction formula $(K_S+R)\cdot R=-2$.
Since $-K_S\cdot R=\deg(R)=3$, we have $R^2=1$, $(-K_S)\cdot (-K_S-R)=2$,
and $(-K_S-R)^2=0$. Hence $|-K_S-R|$ is a pencil of conics.
Conversely, for any conic $C\subset S$ any smooth member $R\in |-K_S-C|$
is a twisted cubic curve.

\begin{slemma}
\label{lemma:cubic}
Let $P_1,\, P_2\in \Points$ are distinct points such that $P_1\approx P_2$.
Then there exists a unique smooth rational cubic curve $R$ 
passing through $P_1, P_2$. 
This curve does not pass through any point 
$P\in \Points\setminus \{P_1, P_2\}$.
\end{slemma}

\begin{proof}
Let $\psi: S\to \PP^2$ be the contraction defined by $|-K_S-C_{P_1}|=|-K_S-C_{P_2}|$
and let $R$ be the inverse image of the line passing through the points $\psi(P_1)$ 
and $\psi(P_2)$. Then $-K_S\cdot R=\psi^*(-K_{\PP^2})\cdot R=3$.
Thus $R$ is a smooth rational cubic curve 
passing through $P_1, P_2$. If there is another such a curve $R'$, then 
$-K_{\PP^2}\cdot \psi(R')=\psi^*(-K_{\PP^2})\cdot R'=3$, hence $\psi(R')$ is 
a line passing through the points $\psi(P_1)$ 
and $\psi(P_2)$. But then $\psi(R')=\psi(R)$ and so $R'=R$.
\end{proof}

\begin{snotation}
For $P\in \Points$,  let $\{P,\, P_1,\, P_2\}\subset \Points$ be the corresponding equivalence class
of $\approx$. Then by $R_P$ we denote the twisted cubic curve passing through $P_1$ 
and $P_2$ (see 
Lemma~\ref{lemma:cubic}). 
By the construction, the curve $R_P$ uniquely determines the conic bundle $\varphi_P$
and so $R_{P'}=R_{P''}$ for $P', P''\in \Points$ if and only if $P'=P''$.
This shows that there exists a natural bijection between the set of twisted cubic curves
of the form $R_P$ and $\Points$.
Denote 
\[
\DCurves:=\sum_{P\in \Points} R_P.
\]
Then we have 
\[
\DCurves\sim -9K_S.
\]
\end{snotation}

\begin{scorollary}
\label{cor:cubic1}
Let $R_P$ be a twisted cubic as above.
\begin{enumerate}
\item
\label{cor:cubic1:a}
For any line $L\subset S$ we have 
\begin{equation}
\label{eq:cor:cubic1:L}
R_P\cdot L=
\begin{cases}
1 &\text{if $L$ lies in the fibers of $\varphi_P$,} 
\\
0 &\text{if $L$ is a section of $\varphi_P$.} 
\end{cases} 
\end{equation} 
In particular, the intersections of components $\DCurves$ and $\DLines$
are transversal. Moreover, the intersections of $\DCurves$ and $\DLines$ 
are transversal outside $\Points$.
\item
\label{cor:cubic1:b}
For any conic $C\subset S$ we have 
\begin{equation}
\label{eq:cor:cubic1:C}
R_P\cdot C=
\begin{cases}
2 &\text{if $C+R_P\sim -K_S$,} 
\\
1 &\text{if $C+R_P\not\sim -K_S$.} 
\end{cases} 
\end{equation} 
\end{enumerate}
\end{scorollary}

\begin{scorollary}
\label{cor:cubic2}
Let $P\in \Points$ and let $\{P,\, P_1,\, P_2\}\subset \Points$ be the corresponding equivalence class 
of $\approx$. Then the following holds:
\begin{enumerate}
\item
\label{cor:cubic2a}
$R_P\cap \Points= \{P_1,\, P_2\}$ and 
$\DLines\cap R_{P}=\{P_1,\, P_2\} \cup (R_P\cap C_P)$, where 
the intersection $R_P\cap C_P$
consists of two points lying on different components of $C_P$.
\item
\label{cor:cubic2b}
There exist exactly two components $R',\, R''\subset \DCurves$
passing through~$P$ and $R'\sim R''\sim -K_S-C_P$.

\item
\label{cor:cubic2d}
$(\DLines\cap \DCurves)\setminus \Points$ consists of at most $30$ points.

\end{enumerate}
\end{scorollary}

\begin{proof}
The assertion 
\ref{cor:cubic2a} follows from \eqref{eq:cor:cubic1:L} and Lemma~\ref{lemma:cubic}.
For \ref{cor:cubic2b} we let $R$ be a component of $\DCurves$
passing through~$P$. 
Since $R\cdot C_P\ge 2$, we have $R\in |-K_S-C_P|$ by Corollary~\ref{cor:cubic1}.
Then either $R\ni P_1$ or $R\ni P_2$,
i.e. $R=R_{P_2}$ or $R=R_{P_1}$. 
This proves \ref{cor:cubic2b}.
Further, since $\DLines\cdot\DCurves=(-2K_S)\cdot (-9K_S)=90$
and the components of $\DLines$ and $\DCurves$ meet each other transversally (see~\eqref{eq:cor:cubic1:L}), 
the number of intersection points $\DLines\cap \DCurves$ outside $\Points$ is at most
$90-15\cdot 2\cdot 2=30$ (see Corollary~\ref{cor:cubic1}\ref{cor:cubic1:a}). This proves \ref{cor:cubic2d}.
\end{proof}

\subsection{$\Sym_5$-invariant pencil}
A brief computation with characters (see e.g. \cite{Shepherd-Barron-S5}) shows that the 
decomposition of space $H^0(S,-2K_S)$ as an $\Sym_5$-module
contains exactly one copy of $W_1$ and exactly one copy of $W_1'$.
Hence the linear system $|-2K_S|$ contains an $\Sym_5$-invariant pencil $\PPP_{\mathrm{W}}$.
The action of $\Sym_5$ on $\PPP_{\mathrm{W}}$ is non-trivial and there are two invariant members.
One of them coincides with $\DLines$ (see \eqref{eq:DLines}).
Let $Z_{\mathrm{W}}\in \PPP_{\mathrm{W}}$ be an $\Sym_5$-invariant divisor other than $\DLines$.
One can show that $Z_{\mathrm{W}}$ is not a union of conics. 
Indeed, otherwise there is an $\Sym_5$-orbit $\{C_1,\dots,C_5\}$, where the $C_i$'s are distinct conics
and then the linear system $|C_1|$ defines a $\Sym_4$-equivariant conic bundle $S\to \PP^1$
having an invariant fiber. But in this case, the $\Sym_4$-action on the base must be trivial and 
the singular points of degenerate fibers must be fixed by $\Sym_4$.
Since $\Sym_4$ does not have two-dimensional faithful representations, this is impossible.
Hence, $Z_{\mathrm{W}}$ is irreducible.
By the adjunction formula $\p(Z_{\mathrm{W}})=6$. Since $\Sym_5$ cannot faithfully act on a smooth curve of genus $\le 2$,
the number of singular points of $Z_{\mathrm{W}}$ is a most $4$. If $Z_{\mathrm{W}}$ is singular, this implies 
that $\Sym_5$ has a fixed point $P\in S$. But then we have a faithful two-dimensional representation of $\Sym_5$
on the tangent space $T_{S,P}$, a contradiction.
Thus $Z_{\mathrm{W}}$ is smooth curve of genus $6$.
If we regard $S$ as the blowup $\psi: S\to \PP^2$ of four points on a plane, then
$\psi(Z_{\mathrm{W}})$ is a curve of degree six with four nodes. It is called Wiman's sextic 
(see \cite[\S~8.5.4]{Dolgachev-ClassicalAlgGeom} and references therein).

\begin{lemma}
\label{lemma:S5}
\begin{enumerate}
\item 
\label{lemma:S5:1}
The action of $\Sym_5$ on $\Lines$ is transitive, 
the stabilizer $\St_{\Sym_5}(L)$ of an element $L\in \Lines$ is the 
group $\Sym_3\times \Sym_2\simeq \Dih_6$, which is the normalizer of a Sylow 
$3$-subgroup $\Syl_3(\Sym_5)$, and the correspondence 
\[
\Lines\longrightarrow \{\text{Sylow $3$-subgroups}\},\quad
L \longmapsto \Syl_3\big(\St_{\Sym_5}(L)\big)
\]
is bijective. 
Moreover, 
the kernel of the homomorphism $\St_{\Sym_5}(L) \longrightarrow 
\Aut(L)$ is generated by a transposition
and the correspondence 
\[
\Lines\longrightarrow \{\text{transpositions in 
$\Sym_5$}\},\quad
L \longmapsto \left[\text{generator of $\ker \big(\St_{\Sym_5}(L) \longrightarrow 
\Aut(L)\big)$}\right]
\]
is also bijective.

\item 
\label{lemma:S5:2}
The action of $\Sym_5$ on $\Points$ is transitive, the 
stabilizer $\St_{\Sym_5}(P)$ of an element $P\in \Points$ is a Sylow 
$2$-subgroup in $\Sym_5$, and the correspondence 
\[
\Points\longrightarrow \{\text{Sylow $2$-subgroups}\},\quad
P \longmapsto \St_{\Sym_5}(P)
\]
is bijective. Moreover, the correspondence 
\[
\Points\longrightarrow \{\text{double transpositions in 
$\Sym_5$}\},\quad
P \longmapsto \left[\text{generator of $\z (\St_{\Sym_5}(P))$}\right]
\]
is also bijective.
\end{enumerate}
\end{lemma}

\begin{proof}
\ref{lemma:S5:1}
Let $\Omega\subset \Lines$ be the smallest 
orbit, let $L\in \Omega$, and let $H:=\St_{\Sym_5}(L)$.
Let $\varphi : S\to S'$ be the contraction of $L$ and let $P':=\varphi(L)$.
Then the group $H$ faithfully acts on the tangent space $T_{S',\varphi(L)}$,
hence $H$ has a faithful $2$-dimensional representation. On the other hand, 
$|H|\ge 12$, hence
$H$ is isomorphic to either $\Alt_5$, $\Sym_4$, $\Sym_3\times\Sym_2$ or 
$\Alt_4$. This implies that the only the last possibility occurs and 
$\Omega =\Lines$. Finally, $\ker(\St_{\Sym_5}(L)\to \Aut(L))$ consists of 
elements acting by scalar matrices on $T_{S',\varphi(L)}$, so it coincides
with the center of $\St_{\Sym_5}(L)$.

\ref{lemma:S5:2}
As above, let $\Omega\subset \Points$ be the smallest 
orbit, let $P\in \Omega$, and let $H:=\St_{\Sym_5}(P)$.
Then $|H|\ge 8$, hence the group $H$ is isomorphic to either $\Alt_5$, 
$\Sym_4$, $\mumu_5 
\rtimes \mumu_4$ or $\Dih_4$. 
On the other hand, $H$ has a faithful $2$-dimensional representation on the 
tangent 
space of $S$ at $P$.
Then $H\simeq\Dih_4$ and $\Omega 
=\Points$.
\end{proof}

\begin{scorollary}
The base locus $\Bs \PPP_{\mathrm{W}}$ of the $\Sym_5$-invariant pencil $\PPP_{\mathrm{W}}\subset |-2K_S|$
consists of $20$~distinct points. In particular, $Z_{\mathrm{W}}\cap \Xi=\varnothing$ and 
the intersection $Z_{\mathrm{W}}\cap L$, where $L$ is a line, consists of $2$~distinct points.
\end{scorollary}

\begin{proof}
Since the stabilizer of a point on a smooth curve is cyclic, $Z_{\mathrm{W}}\cap \Xi=\varnothing$
by Lemma~\ref{lemma:S5}\ref{lemma:S5:1}. Since some transposition acts on $L$ trivially,
the intersection $Z_{\mathrm{W}}\cap L$ is transversal. Hence $Z_{\mathrm{W}}\cap L$
consists of $2$~distinct points.
\end{proof}

\begin{lemma}[see e.g. {\cite[\S~6.3]{Dolgachev-Iskovskikh}}]
\label{lemma:S5AutS5}
As $\Sym_5$-module, $\Pic(S)\otimes \CC$ is isomorphic to $W_4\oplus W_1$. 
\end{lemma}

\begin{scorollary}
\label{cor:Lef}
The Lefschetz number 
\[
\Lef (\sigma,S):= \sum (-1)^k \Tr_{H^k(S,\RR)}(\sigma)
\]
of elements of $\Aut(S)$ are given by the fifth 
column of Table~\xref{tab:Lef}.
\end{scorollary}
\begin{table}[H]
\begin{tabular}{l|ccccll}
representative $\sigma\in 
\mathcal{C}$&$\#\mathcal{C}$&$|\sigma|$&$\Tr_{W_4}(\sigma)$&$\Lef(
\sigma, S)$& $\Fix(\sigma, S)^{(1)}$& $\Fix(\sigma, S)^{(0)}$
\\\hline
(1) & 1 & 1 & 4 & 7
\\
(1,2) & 10 & 2 & 2 & 5 & line& 3 points
\\
(1,2)(3,4) & 15 & 2 & 0 &3 & twisted cubic & point
\\
(1,2,3) & 20 & 3 &1 &4& &4 points $\subset \DLines$
\\
(1,2,3)(4,5) & 20 & 6 & $-1$ & 2& &2 points (on a line)
\\
(1,2,3,4) & 30 & 4 & 0 &3& &3 points
\\
(1,2,3,4,5) & 24 & 5 & $-1$& 2&& 2 points
\end{tabular}
\caption{Action of $\Sym_5$ on $\Pic(S)$ and the set of fixed points}
\label{tab:Lef}
\end{table}

\begin{sclaim}
\label{claim:transv}
For distinct elements $\sigma_1,\, \sigma_2\in \Aut(S)$ of order $2$
the curves $\Fix(\sigma_1, S)^{(1)}$ and $\Fix(\sigma_2, S)^{(1)}$
meet each other transversally.
\end{sclaim}

\begin{proof}
Indeed, if $P$ is a point of tangency, then $\sigma_1$ and $\sigma_2$ have a common invariant vector 
in the tangent space $T_{S,P}$, hence $\sigma_1\comp \sigma_2^{-1}$ is the identity transformation.
\end{proof}

\begin{sclaim}
If $\sigma\in \Sym_5$ is an element of order $2$, then $\dim \Fix(\sigma,S)=1$.
\end{sclaim}

\begin{proof}
Assume that $\sigma$ has only isolated fixed points. Then by the holomorphic 
Lefschetz fixed point formula the number of $\sigma$-fixed points equals $4$.
On the other hand, By Corollary~\ref{cor:Lef} this number must be equal to 
either $3$ or $5$.
\end{proof}

\begin{sclaim}
\label{claim:order3-5}
If $\sigma\in \Sym_5$ is an element of order $3$ or $5$, then $\dim 
\Fix(\sigma,S)=0$.
\end{sclaim}

\begin{proof}
Assume that $\Fix(\sigma,S)$ contains an irreducible curve $\Gamma$.
Then $\Gamma\cap \Points=\varnothing$ and $\Gamma$ is not a line by 
Lemma~\ref{lemma:S5}\ref{lemma:S5:2}. Hence, there is at most 
one line 
passing through any point $P\in \Gamma$ and so any line meeting $\Gamma$ must be
$\sigma$-invariant. On the other hand,
$\Gamma$ is a nef divisor and so there exists at least $4$ lines meeting 
$\Gamma$.
This contradicts~\ref{lemma:S5}\ref{lemma:S5:1}.
\end{proof}

\begin{sclaim}
\label{claim:conic}
For any non-identity element $\sigma\in \Sym_5$ the set $\Fix(\sigma,S)$ does 
not contain any conic. 
\end{sclaim}

\begin{proof}
Assume that $\sigma$ acts trivially on a conic $C$.
Then $C$ must be smooth. Recall that there are exactly four lines 
$L_1,\dots,L_4$ on $S$ meeting $C$ and these lines are disjoint (see~\ref{conic:bundles}).
Therefore, there is a contraction $\varphi: S\to \PP^2$ of 
them so that $\varphi(C)$ is a conic on $\PP^2$. Since the contraction 
$\varphi$ is $\sigma$-equivariant, the induced action on $\PP^2$ is trivial, a 
contradiction.
\end{proof}

\begin{lemma}
\label{lemma:S5transpositions}
Let $\sigma\in \Sym_5$ be a transposition. Then the following assertions hold.
\begin{enumerate}
\item
\label{lemma:S5transpositions1}
The set $\Fix(\sigma,S)^{(1)}$ is a 
line on $S$. 
\item
\label{lemma:S5transpositions2}
There are exactly four 
$\sigma$-invariant lines; these are $\Fix(\sigma,S)^{(1)}$ and three lines 
$L_1,\, L_2,\, L_3$ meeting $\Fix(\sigma,S)^{(1)}$.
\item\label{lemma:S5transpositions3}
The set $\Fix(\sigma,S)^{(0)}$ is a union of three points
$P_1\in L_1$,\, $P_2\in L_2$,\, $P_3\in L_3$, and $P_i\notin\Points$. 
\end{enumerate}
In particular, $\Fix(\sigma,S)\subset \DLines$.
\end{lemma}

\begin{proof}
By 
Lemma~\ref{lemma:S5}\ref{lemma:S5:1} the set $\Fix(\sigma,S)^{(1)}$ contains a 
line, say $L$. The lines 
$L_1,\, L_2,\, L_3$ meeting $L$ are $\sigma$-invariant but $L_i\not \subset 
\Fix(\sigma,S)^{(1)}$ since $\Fix(\sigma,S)^{(1)}$ must be a smooth curve.
Now Lemma~\ref{lemma:S5}\ref{lemma:S5:1} implies that $L,\, L_1,\, L_2,\, L_3$
are the only $\sigma$-invariant lines on $S$ because there are exactly four 
involutions commuting with $\sigma$. This proves \ref{lemma:S5transpositions2}.
Note that the lines $L_1,\, L_2,\, L_3$ are disjoint, hence there is a 
$\sigma$-equivariant contraction $\varphi: S\to \PP^1\times \PP^1$ of 
them so that $\varphi(L)$ is the diagonal and the induced action on 
$\PP^1\times \PP^1$ is an involution interchanging the factors. 
Then $\Fix(\sigma, \PP^1\times \PP^1)=\varphi(L)$, hence $\Fix(\sigma,S)$ is 
contained in $\varphi^{-1}(\varphi(L))= L\cup L_1\cup L_2\cup L_3$.
Since the variety $\Fix(\sigma,S)$ is smooth, we see that $L$ is the only 
one-dimensional component of $\Fix(\sigma,S)$, and $\Fix(\sigma,S)\setminus L$
consists of three isolated points lying on $L_1,\, L_2,\, L_3$. 
This proves~\ref{lemma:S5transpositions1} and ~\ref{lemma:S5transpositions3}.
\end{proof}

\begin{lemma}
\label{lemma:S5doubletranspositions}
Let $\sigma=\sigma_1\comp 
\sigma_2\in \Sym_5$, where $\sigma_1$, $\sigma_2$ are disjoint 
transpositions. Then the following assertions hold.
\begin{enumerate}
\item
\label{lemma:S5doubletranspositions1}
There are exactly two
$\sigma$-invariant lines; these are 
$\Fix(\sigma_1,S)^{(1)}$ and $\Fix(\sigma_2,S)^{(1)}$.

\item
\label{lemma:S5doubletranspositions2}
$\Fix(\sigma,S)^{(0)}=\Fix(\sigma_1,S)^{(1)}\cap 
\Fix(\sigma_2,S)^{(1)}\in \Points$.
\item
The set $\Fix(\sigma,S)^{(1)}$ is the
twisted cubic curve described in Lemma~\xref{lemma:cubic}.
\end{enumerate}
\end{lemma}

\begin{proof}
Since the size of $\Points$ is odd, there exists a $\sigma$-invariant point 
$P\in \Points$. Then the pencil $|C_P|$ defines a $\sigma$-equivariant conic 
bundle $\varphi_P: S\to \PP^1$. 
Put $R:=\Fix(\sigma,S)^{(1)}$.
Any component of $R$ is not a line 
nor a conic, hence $R$ is not contained in the fibers of $\varphi_P$.
This implies that the action $\sigma$ on the base of $\varphi_P$ is trivial.
The conic bundle $\varphi_P$ has three degenerate fibers 
$C_P$, $C_{P'}$, and $C_{P''}$ whose singular points 
$P$, 
$P'$, and $P''$ are fixed by $\sigma$.
The curve $R$ meets any smooth fiber at two points, hence $\varphi_P|_R: R\to 
\PP^1$
is a double cover. If $R$ is reducible, then $R$ is a disjoint union of two 
smooth rational curves. This contradicts our computation $\Lef(\sigma, S)=3$.
Hence $R$ is irreducible (and smooth). The ramification divisor of $\varphi_P|_R: 
R\to \PP^1$ is contained in the set of the intersection points $\{P, P', 
P''\}$. Since the degree of this divisor is even, up to permutations, we may 
assume that $P'$ and 
$P''$ are ramification points, and $P\notin R$. Then by the Hurwitz formula, $R$ 
is rational and so 
$\Fix(\sigma,S)^{(0)}=P$. This proves \ref{lemma:S5doubletranspositions2}. 
Since $P$ is an isolated fixed point, the action on $\sigma$ on the tangent 
space is given by the minus identity matrix.
Hence, the lines $L_1$ and $L_2$ are invariant. Similar arguments show that 
the lines $L_1'$ and $L_2'$ (resp. $L_1''$ and $L_2''$) are interchanged by 
$\sigma$. Since $\sigma_i$ commute with $\sigma$, we have $\sigma_i(P)=P$.
By Lemma~\ref{lemma:S5transpositions}\ref{lemma:S5transpositions3} we conclude 
that $\Fix(\sigma_i,S)^{(1)}$ is a line passing through $P$.
Thus, up to permutation, $\Fix(\sigma_1,S)^{(1)}=L_1$ and 
$\Fix(\sigma_2,S)^{(1)}=L_2$. 
By Lemma~\ref{lemma:S5transpositions}\ref{lemma:S5transpositions1} these are 
the only $\sigma$-invariant lines.
This proves \ref{lemma:S5doubletranspositions1}.

Now we claim that $-K_S-R$ is a nef divisor.
Assume the converse. Then $(-K_S-R)\cdot L<0$ for some line $L$.
Thus $R\cdot L\ge 2$. Take a point $Q\in R\cap L$. If $Q\in \{P,P',P''\}$, then
$L\in \{L_1,L_2,L_1',L_2',L_1'',L_2''\}$ and 
$R\cdot L=1$ by the arguments above. Thus $Q\notin \{P,P',P''\}$ and, in 
particular, $Q\notin\Points$. Hence $L$ is the only line passing through 
$Q\in R$ and so $L$ is $\sigma$-invariant. This contradicts 
\ref{lemma:S5doubletranspositions1}. Therefore, $-K_S-R$ is a nef and, 
in particular, $-K_S-R\sim F$, where $F$ is effective. On the other hand, 
$F\cdot (-K_S)=(-K_S-R)\cdot (-K_S)\le 2$ and 
$F\cdot (C_P)=0$, hence $F$ is a conic and $F\sim C_P$ by 
(see~\ref{conic:bundles}). Thus $R=R_{P',P''}$ (see Lemma~\ref{lemma:cubic}).
\end{proof}

\begin{scorollary}
\label{cor:isolated:fix-p}
Let $P\in S$ be an isolated fixed point of an element $\sigma\in \Sym_5$ of order $2$.
Then $P\in (\DLines\cap \DCurves)$. Moreover, one of the following holds:
\begin{enumerate}
\item
$P\in \Points$ and $\St_{\Sym_5}(P)$ is a Sylow 
$2$-subgroup in $\Sym_5$;
\item
$\sigma$ is a transposition, $P\in (\DLines\cap \DCurves)\setminus \Points$, 
$|\St_{\Sym_5}(P)|=4$, and $\St_{\Sym_5}(P)$
is the centralizer of a double transposition.
\end{enumerate}
\end{scorollary}

\begin{proof}
Assume that $P\notin \Points$.
By Lemma~\ref {lemma:S5doubletranspositions}\ $\sigma$ is a transposition and 
by Lemma~\ref {lemma:S5transpositions} $P$ lies on a (unique) line $L$, where $L=\Fix(\sigma_1, 
S)^{(1)}$ for another transposition $\sigma_1$. 
The line $L$ is $\St_{\Sym_5}(P)$-invariant, hence the group $\St_{\Sym_5}(P)$ is abelian.
Then by Table~\ref{table:subgroupsS5} the group $\St_{\Sym_5}(P)$ is generated by
$\sigma$ and $\sigma_1$, and is isomorphic to $\mumu_2\times \mumu_2$.
Finally, the element $\sigma\comp \sigma_1$ is a double transposition, hence
$P\in \DCurves$ by Lemma~\ref {lemma:S5doubletranspositions}.
\end{proof}

\begin{scorollary}
\label{cor:LR}
Let $P\in (\DLines\cap \DCurves)\setminus \Points$. Then 
the following assertions hold.
\begin{enumerate}
\item 
\label{cor:LR:b}
$\St_{\Sym_5}(P)$ is isomorphic to $\mumu_2\times \mumu_2$;

\item 
\label{cor:LR:a}
$P$ is an isolated fixed point of a 
transposition $\sigma\in \Sym_5$;

\item 
\label{cor:LR:c}
$(\DLines\cap \DCurves)\setminus \Points$ is an orbit of $\Sym_5$ of length $30$;
\item
\label{cor:LR:d}
the intersection $L\cap \DCurves$ is transversal outside $\Points$.
\end{enumerate}
\end{scorollary}

\begin{proof}
In this case $P$ lies on a unique line, say $L$. Hence the group $\St_{\Sym_5}(P)$ is abelian
and contains a transposition $\sigma_1$.
Further, $P$ lies on a twisted cubic $R$, hence $\St_{\Sym_5}(P)$ contains 
a double transposition $\sigma_2$. 
Then $\St_{\Sym_5}(P)$ is isomorphic to $\mumu_2\times \mumu_2$ by Table~\ref{table:subgroupsS5}.
This proves \ref{cor:LR:b}.
Considering the induced representation on $T_{S,P}$ we obtain that $P\in \Fix(\sigma,S)^{(0)}$
for some element $\sigma\in \St_{\Sym_5}(P)$.
This proves \ref{cor:LR:a}.
The assertions \ref{cor:LR:c} and \ref{cor:LR:d} follow from Corollary~\ref{cor:cubic2}\ref{cor:cubic2d}.
\end{proof}

\begin{scorollary}
\label{cor:RR}
Any two distinct irreducible components $R',\, R''\subset \DCurves$
meet each other transversally. Moreover, the following assertions hold.
\begin{enumerate}

\item \label{cor:RR:a}
If $R'\sim R''$, then $R'\cap R''=\{P\}$, where $P\in \Points$.

\item \label{cor:RR:b}
If $R'\not \sim R''$, then $R'\cap R''=\{P_1,\, P_2\}$, where 
$P_1,\, P_2\notin\DLines$.
\end{enumerate}
\end{scorollary}

\begin{proof}
The first assertion follows from Claim~\ref{claim:transv}. Note that 
$R'\cdot R''=1$ if $R'\sim R''$ and $R'\cdot R''=2$ if $R'\not\sim R''$.
Then \ref{cor:RR:a} and \ref{cor:RR:b} follow from Lemma~\ref{lemma:cubic}
and Corollary~\ref{cor:LR}\ref{cor:LR:d}, respectively.
\end{proof}

\begin{lemma}
\label{lemma:order5}
Let $\sigma\in \Sym_5$ be an element of order $5$. Then 
$\Fix(\sigma,S)$ consists of 
two points $P_1,\, P_2$. These points lie in $\Sing(\DCurves)\setminus \Points$ and 
$\St_{\Sym_5}(P_i)\simeq \Dih_5$.
\end{lemma}

\begin{proof}
By Claim~\ref{claim:order3-5} the set $\Fix(\sigma,S)$ is zero-dimensional,
hence by Corollary~\ref{cor:Lef} it consists of 
two points. 
Let $\Fix(\sigma,S)=\{P_1,\, P_2\}$. By 
Lemma~\ref{lemma:S5}\ref{lemma:S5:2}
$P_i\notin \Points$ and by~\ref{lemma:S5}\ref{lemma:S5:1} $P_i$ does not lie on 
a line. Let $H$ be the group generated by $\sigma$. Then 
its normalizer $\N_{\Sym_5}(H)$ is a group of order $20$ isomorphic to 
$\mumu_5 \rtimes \mumu_4$. The set $\{P_1,\, P_2\}$ is 
$\N_{\Sym_5}(H)$-invariant. On the other hand, $\N_{\Sym_5}(H)$ has no 
two-dimensional faithful representation, so it cannot act on the tangent 
space $T_{S, P_i}$. Therefore, $\N_{\Sym_5}(H)\not \subset \St_{\Sym_5}(P_i)$
and $\St_{\Sym_5}(P_i)$ contains the group 
$\N_{\Alt_5}(H)\subset \N_{\Sym_5}(H)$ (see 
Table~\ref{table:subgroupsS5}). Again from Table~\ref{table:subgroupsS5} we see 
that $\St_{\Sym_5}(P_i)=\N_{\Alt_5}(H)$.
Hence $\St_{\Sym_5}(P_i)$ contains $5$ elements $\tau_1,\dots, \tau_5$ of order 
$2$, which are double transpositions. Therefore, $P_i\subset \Fix(\tau_i,S)$.
Since $P_i\notin \Points$, we have $P_i\in \Fix(\tau_j,S)\subset \DCurves$ for all $j=1,\dots,5$ 
(see Lemma~\ref{lemma:S5doubletranspositions}). Thus $P_i\in \Sing(\DCurves)\setminus \Points$.
\end{proof}

\begin{lemma}
Let $\Omega\subset Z_{\mathrm{W}}$ be an orbit of $\Sym_5$ of length $<120$ and let $P\in \Omega$. 
Then one of the following holds:
\begin{enumerate}
\item 
$|\Omega|=20$, $\Omega=\Bs \PPP_{\mathrm{W}}=Z_{\mathrm{W}}\cap \DLines$, and $|\St_{\Sym_5}(P)|=6$;
\item 
$|\Omega|=30$, $\Omega\subset\DCurves\cap Z_{\mathrm{W}}$, $\Omega\cap \DLines=\varnothing$, and $|\St_{\Sym_5}(P)|=4$;
\item 
$|\Omega|=60$, $\Omega\subset\DCurves\cap Z_{\mathrm{W}}$, $\Omega\cap \DLines=\varnothing$, $|\St_{\Sym_5}(P)|=2$;
\end{enumerate}
Moreover, $\Omega_{30}\cup\Omega_{60}=\DCurves\cap Z_{\mathrm{W}}$.
\end{lemma}
\begin{proof}
The stabilizer of any point $P\in \Omega$ is a cyclic group. Hence its order $n_P$ can take the values
$2\le |\St_{\Sym_5}(P)|\le 6$. Let $m_n$ be the number of orbits of length $n$ and let $Z':=Z_{\mathrm{W}}/\Sym_5$. Then by the Hurwitz formula 
\[
\frac{1}{12}=\frac{2\g(Z_{\mathrm{W}}) -2}{|\Sym_5|}=2\g(Z') -2+ \sum_{n=2}^{6} m_n \left(1-\frac {1}{n}\right) 
\]
Note that $m_2>0$ because $Z_{\mathrm{W}}$ meets $\DLines$ and $\DCurves$. 
Then $\g(Z')=0$ and easy computations give us the only possibility: $m_2=m_4=m_6=1$, $m_3=m_5=0$.
Let $\Omega_{20}$, $\Omega_{30}$, and $\Omega_{60}$ be orbits of length $20$, $30$, and $6$, respectively.
Clearly, $\Omega_{20}=Z_{\mathrm{W}}\cap \DLines$ and so $\Omega_{30}\cap \DLines=\Omega_{60}\cap \DLines=\varnothing$. Hence for $P\in \Omega_{30}\cup\Omega_{60}$, the stabilizer of $P$ does not 
contain transpositions by Lemma~\ref{lemma:S5transpositions}. 
Then $\Omega_{30}\cup\Omega_{60}\subset \DCurves$ by Lemma~\ref{lemma:S5doubletranspositions}
and $\Omega_{30}\cup\Omega_{60}=\DCurves\cap Z_{\mathrm{W}}$ because $\DCurves\cap Z_{\mathrm{W}}=90$.
\end{proof}

\begin{lemma}
\label{lemma:order3}
Let $\sigma\in \Sym_5$ be an element of order $3$. Then 
$\Fix(\sigma,S)$ consists of 
four points $P_1,P_2,Q_1,Q_2$ so that
\begin{enumerate}
\item 
\label{lemma:order3a}
$P_1,P_2$ lie on a line, $P_1,P_2\notin \Points$, and 
$\St_{\Sym_5}(P_i)\simeq \mumu_6$;
\item 
\label{lemma:order3b}
$Q_1,Q_2 \in \DCurves$, $Q_1,Q_2 \notin \DLines$, and $\St_{\Sym_5}(Q_i)\simeq \Sym_3$, 
$\St_{\Sym_5}(Q_i)\subset\Alt_5$.
\end{enumerate}
\end{lemma}

\begin{proof}
By Claim~\ref{claim:order3-5} the set $\Fix(\sigma,S)$ is zero-dimensional,
hence by Corollary~\ref{cor:Lef} it consists of 
four points. Clearly, there is a $\sigma$-invariant line $L$ and there are 
two fixed points $P_1,P_2$ lying on $L$. By 
Lemma~\ref{lemma:S5}\ref{lemma:S5:2}\ $P_i\notin \Points$.
Let $L_1,L_2,L_3$ be lines meeting $L$.
They are disjoint and permuted by $\sigma$. On the other hand, any 
line $L'$ different from $L$ meets exactly one of the $L_1,L_2,L_3$.
Hence, $L$ is the only invariant line and $\Fix(\sigma, S)\cap \DLines=\{P_1,P_2\}$.
Let $H$ be the group generated by $\sigma$. Then 
its normalizer $\N_{\Sym_5}(H)$ is a group $\Sym_3\times \Sym_2\simeq \Dih_6$. Its center is 
generated by a transposition $\tau$, $\N_{\Sym_5}(H)=\C_{\Sym_5}(\tau)$, and $\N_{\Sym_5}(H)$ 
contains three more transpositions $\tau_1,\tau_2,\tau_3$.
By Lemma~\ref{lemma:S5}\ref{lemma:S5:1} the line $L$ is 
$\N_{\Sym_5}(H)$-invariant and the action of $\tau$ on $L$ is trivial.
Thus $\St_{\Sym_5}(P_i)\supset \langle \sigma, \tau\rangle\simeq \mumu_6$.
Suppose that $\St_{\Sym_5}(P_i)=\N_{\Sym_5}(H)$.
Since $P_i\notin \Points$, $P_i$ is an isolated fixed point for $\tau_j$, 
$j=1,2,3$. Therefore, the induced action of $\tau_j$ on $T_{S, P_i}$ is given 
by the scalar matrix and so $\tau_1\comp \tau_2^{-1}$ is trivial, a 
contradiction.
Thus $\St_{\Sym_5}(P_i)=\langle \sigma, \tau\rangle$.
This proves \ref{lemma:order3a}.

We have two more $\sigma$-fixed points $Q_1,Q_2\in S\setminus \DLines$.
The set $\{Q_1,\, Q_2\}$ is 
$\N_{\Sym_5}(H)$-invariant. On the other hand, $Q_1,Q_2\notin \Fix(\sigma\comp 
\tau, S)$ because $P_1$ and $P_2$ are the only points fixed by $\sigma\comp 
\tau$ (see Corollary~\ref{cor:Lef}). Thus $\tau \notin \St_{\Sym_5}(Q_i)$. 
Therefore, $\St_{\Sym_5}(Q_i)$ is of index $2$ in $\N_{\Sym_5}(H)$, so $\St_{\Sym_5}(Q_i)= 
\N_{\Alt_5}(H)$. This group contains three 
double transpositions $\tau_1',\tau_2',\tau_3'$ and $Q_i\in \Fix(\tau_j', S)$, 
$j=1,2,3$. Since $Q_i\notin \Points$, we have 
$Q_i\in \Fix(\tau_j',S)^{(1)}$ (see Lemma~\ref{lemma:S5doubletranspositions}).
Hence, $Q_1,Q_2 \in \DCurves$. 
This proves~\ref{lemma:order3b}.
\end{proof}

\begin{scorollary}
\label{cor:order3}
In the above notation $\sigma$ acts in $T_{S,Q_i}$ so that $\det(\sigma)=1$ 
and it acts in $T_{S,P_i}$ so that $\det(\sigma)\neq 1$.
\end{scorollary}

\begin{proof}
Let $\bar S:=S/\langle\sigma\rangle$ and let $\pi:S\to \bar S$ be the quotient morphism.
In the case \ref {lemma:order3}\ref{lemma:order3b} we have a representation $\Sym_3 \hookrightarrow \GL(T_{S,Q_i})$.
Then it is easy to see that an element $\sigma\in \Sym_3$ of order $3$
must act on $T_{S,Q_i}$ (is a suitable basis) via 
$\left(\begin{smallmatrix}
\bzeta_3&0\\0 &\bzeta_3^2
\end{smallmatrix}\right)$
Then the point $\pi(Q_i)$ is Du Val of type \type{A_2}.
In the case \ref {lemma:order3}\ref{lemma:order3a} the points $P_1$ and $P_2$
are interchanged by $\St_{\Sym_5}(L)=\N_{\Sym_5}(\langle \sigma\rangle)$.
Therefore, the actions of $\sigma$ on $T_{S,P_1}$ and $T_{S,P_2}$ have the same type.
Note that the morphism $\pi$ is \'etale in codimension one, hence $\bar S$ is a del Pezzo surface with $K_{\bar S}^2=K_S^3/3=5/3$.
Then the points $\pi(P_i)$ cannot be Du Val, and so 
act on $T_{S,P_i}$ via a scalar matrix.
\end{proof}

\begin{lemma}
\label{lemma:order4}
Let $\sigma\in \Sym_5$ be an element of order $4$. Then 
the set $\Fix(\sigma,S)$ consists of three points $P_0$, $Q_1$, $Q_2$
so that 
\begin{enumerate}
\item \label{lemma:order4a}
$P_0\in \Points$ and 
$\St_{\Sym_5}(P_0)\simeq \Dih_4$;
\item \label{lemma:order4b}
$Q_1,Q_2 \in \DCurves$, $Q_1,Q_2 \notin \DLines$, and 
$\St_{\Sym_5}(Q_i)=\langle\sigma\rangle$. 
\end{enumerate}
\end{lemma}

\begin{proof}
Let $P$ be a $\sigma$-fixed point.
If $P\in \Points$, then $\St_{\Sym_5}(P)\simeq 
\Dih_4$ by Lemma~\ref{lemma:S5}\ref{lemma:S5:2}. So we assume that $P\notin 
\Points$. Since by Lemma~\ref{lemma:S5}\ref{lemma:S5:1} there are no 
$\sigma$-invariant lines, $P\notin \DLines$. Hence $P\in \Fix(\sigma^2,S)^{(1)}$ 
and so $P\in \DCurves$ by Lemma~\ref{lemma:S5doubletranspositions}. By 
Lemma~\ref{lemma:S5transpositions} the group $\St_{\Sym_5}(P)$ 
does not contain any transpositions, hence 
$\St_{\Sym_5}(P)=\langle\sigma\rangle$. 
\end{proof}

\begin{proposition}
\label{prop:S_5:fixed-p}
Any non-regular orbit $\Omega\subset S$ of $\Sym_5$ is contained in $\DLines\cup \DCurves$.
Let $\Omega\subset S$ be an orbit of length $<60$ and let $P\in \Omega$. 
Then $\Omega$ is described in Table~\xref{table:S5}.
\begin{table}[h]
\renewcommand{\arraystretch}{1.1}
\setlength{\tabcolsep}{0.3em}
\begin{tabularx}{0.9\textwidth}{|l|l|X|l|l|l|}
\hline
&$|\Omega|$&{$\Omega$}&$\St_{\Sym_5}(P)$,\ $P\in\Omega$& $|\St_{\Sym_5}(P)|$ &\rm $\Alt_5$-orbits 
\\\hline
\nr
\label{table:S5:15}
&$15$& $\Points$& $\Dih_4$ &$8$&$\Omega$
\\
\nr
\label{table:S5:20}
&$20$& $\DLines\cap Z_{\mathrm{W}}=\Bs \PPP_{\mathrm{W}}$
&$\mumu_6$&$6$&$\Omega$
\\
\nr
\label{table:S5:20a}
&$20$& $\Omega\subset \Sing(\DCurves)\setminus \Points$, $\Omega\cap \DLines=\varnothing$& $\Sym_3$&$6$&$\Omega_1\cup 
\Omega_2$
\\
\nr
\label{table:S5:30}
&$30$& $(\DLines\cap \DCurves)\setminus \Points$& $\Sym_2\times\Sym_2$&$4$&$\Omega$
\\
\nr
\label{table:S5:30a}
&$30$& $\Omega\subset\DCurves\cap Z_{\mathrm{W}}$, $\Omega\cap \DLines=\varnothing$ &$\mumu_4$ &$4$&$\Omega$
\\
\nr
\label{table:S5:12}
&$12$& $\Omega\subset \Sing(\DCurves)\setminus \Points$, $\Omega\cap \DLines=\varnothing$& $\Dih_5$&$10$&$\Omega_1\cup \Omega_2$
\\\hline
\end{tabularx}
\caption{Non-regular orbits of $\Aut(S)$} 
\label{table:S5}
\end{table}
\end{proposition}

\begin{proof}
Let $H:=\St_{\Sym_5}(P)$.
Pick a point $P\in \Omega$. 
If $|H|=2$, then $P\in \DLines\cup \DCurves$
by Lemmas~\ref{lemma:S5transpositions} and~\ref{lemma:S5doubletranspositions}.
Thus we assume that $|H|\ge 3$. We consider the possibilities for $H$
according to Table~\ref{table:subgroupsS5}.

\subsubsection*{Case: $H$ contains an element of order $3$.}
By Lemma~\ref{lemma:order3} one of the following holds:
$P\in \DLines\setminus \Points$ and 
$\St_{\Sym_5}(P)\simeq \mumu_6$ or
$P\in \DCurves\setminus \DLines$ and $\St_{\Sym_5}(P)= 
\N_{\Alt_5}(H)$. 

\subsubsection*{Case: $H\simeq \mumu_2\times \mumu_2$.}
Considering the representation of $H$ in the tangent space $T_{S,P}$ we see 
that $P$ is not an isolated fixed point for some element $\sigma\in H$.
By Corollary ~\ref{cor:isolated:fix-p}
$P\in (\DLines\cap \DCurves)\setminus \Points$
and $\St_{\Sym_5}(P)$
is the centralizer of a double transposition.

\subsubsection*{Case: $H$ contains an element of order $4$.}
By Lemma~\ref{lemma:order4} one of the following holds:
$P\in \Points$ or
$P \in \DCurves\setminus \DLines$ and 
$\St_{\Sym_5}(Q_i)=\langle\sigma\rangle$. 

\subsubsection*{Case: $H$ contains an element of order $5$.}
By Lemma~\ref{lemma:order5} $P\in \DCurves\setminus \DLines$ and 
$\St_{\Sym_5}(P_i)=\N_{\Alt_5}(H)$.
\end{proof}

\begin{proof}[Proof of Theorem~\xref{thm:DP5}]
\ref{thm:DP5a} 
Denote $\bar S:=S/\Sym_5$ and let $\pi: S\to \bar S$ be the quotient morphism.
First, we describe the singularities of $\bar S$. Pick a point $P\in S$ 
and let $\bar P=\pi(P)$. The type of the singularity 
$\bar P\in \bar S$ is determined by the induced action of the group $H:=\St_{\Sym_5}(P)$ on the 
tangent space $T_{S,P}$. We may assume that $H$ is not trivial.
Then $P\in \DLines\cup \DCurves$ and so $H$ contains an element $\sigma$
that acts as reflection on $T_{S,P}$.
If $H$ is of order $2$, then $\sigma$ generates $\St_{\Sym_5}(P)$
and $\bar P\in \bar S$ is a smooth point by the Chevalley–Shephard–Todd theorem.
Thus we may assume that $|H|>2$ and $P$ is contained in one of the orbits 
described in Proposition~\ref{prop:S_5:fixed-p}. In the cases \ref{table:S5:15}, \ref{table:S5:20a}, \ref{table:S5:30}, 
and \ref{table:S5:12} the group $H$ is generated by reflections in $\GL(T_{S,P})$.
Hence, $\bar P\in \bar S$ is a smooth again by the Chevalley–Shephard–Todd theorem.
In the case \ref{table:S5:30a} 
the group $H$ is cyclic and and $H\ni\sigma$, hence its generator acts on $T_{S,P}$ via 
$\left(\begin{smallmatrix}
\bi&0\\0 &-1
\end{smallmatrix}\right)$. Then the point $\bar P$ is Du Val of type \type{A_1}.
Finally in the case \ref{table:S5:20} by Corollary~\ref{cor:order3} the generator $H$ 
acts on $T_{S,P}$ via
$\left(\begin{smallmatrix}
\bzeta_6&0\\0 &\bzeta_6^4
\end{smallmatrix}\right)$, hence the point $\bar P$ is Du Val of type \type{A_2}.

We have shown that the singularities of $\bar S$ are Du Val of types \type{A_1} and \type{A_2}.
Furthermore, $\rk \Pic(\bar S)=\rk\Pic(S)^{\Sym_5}=1$. Since $\bar S$ is rational, 
$-K_{\bar S}$ is ample, i.e. $\bar S$ is a del Pezzo surface with $\rk \Pic(\bar S)=1$ 
two singularities of types \type{A_1} and \type{A_2}. 
Such a surface is unique up to isomorphism and isomorphic to $\PP(1,2,3)$
(see e.g. \cite{Miyanishi-Zhang:88}). 

\ref{thm:DP5b} 
Denote $S':=S/\Alt_5$. There is the decomposition
\[
\pi: S \overset{\pi'}\longrightarrow S' \overset{\bar\pi}\longrightarrow \bar S
\]
where $\pi'$ is branched over $\pi'(\DCurves)$ 
and $\bar \pi$ is a double cover branched over $\pi(\DLines)$.
By the Hurwitz formula we have $K_S=\pi'^* K_{S'}+\DCurves$ and so 
\[\textstyle
K_{S'}^2=\frac 1{60} (K_S- \DCurves)^2 =\frac 1{60}(10 K_S)^2=\frac{25}{3}.
\]
To we describe the singularities of $S'$ we look at the points with non-regular orbits.
As in the proof of Theorem~\xref{thm:DP5}, using Table~\xref{table:S5} we obtain that 
the only singularity of $S'$ is the image of the orbit described in \ref{table:S5:20}
and this singularity is of type $\frac13(1,1)$. Thus for the minimal resolution $\tilde S'$
of $S'$ we have $K_{\tilde S'}^2= 25/3-1/3=8$ and $\uprho(\tilde S')=2$. 
Therefore, $\tilde S'$ is the Hirzebruch surface $\FF_3$ and $S'\simeq \PP(1,1,3)$.

\ref{thm:DP5c} 
Let $\bar\DLines:=\pi(\DLines)$. Then $\pi^*\bar \DLines\sim 2 \DLines\sim 4(-K_S)$.
Since $\pi^* K_{\bar S}\sim 12 K_S$, we obtain $\bar \DLines \qq \frac 13 (-K_{\bar S})$.
Thus $\bar \DLines$ is a conic on $\bar S$. 
\end{proof}

\section{The quotient $\PP(V_6)/\Alt_5$}
\label{sect:V6}
In this section we consider the representation $V_6$ of $\tilde \Alt_5$.
\begin{theorem}
\label{prop:V5-V6}
The variety $\PP(V_6)/\Alt_5$ is rational.
\end{theorem}

To prove Theorem \ref{prop:V5-V6} we need some preparations.
\begin{notation}
Let $C=\upsilon_d(\PP^1)\subset \PP^d$ is a rational normal curve of degree $d\ge 3$
and let $Z\subset \PP^d$
be the secant variety of $C$, that is, the union of all bisecant (and tangent) lines to $C$.

It is known that $Z$ is normal and has only log terminal $\QQ$-factorial singularities.
Moreover, $\deg Z=(d-1)(d-2)/2$, the divisor $-K_Z$ is ample (i.e. $Z$ is a Fano variety), and $-K_Z\qq \frac 4{d-2} H|_Z$ (see \cite{Prokhorov-1993b}, \cite{ein_singularities_2020}). 

Let $\sigma: \tilde \PP^d\to \PP^d$ be the blowup of $C$, let $E$ be the $\sigma$-exceptional divisor, and let $H^*:=\sigma^*H$,
where $H$ is the hyperplane class on~$\PP^d$.
Let $\tilde Z\subset \tilde{\PP^d}$ be the proper transform of~$Z$.
The variety $\tilde Z$ is 
smooth and has $\PP^1$-bundle structure $\pi: \tilde Z \to\Sy^2(C)=\PP^2$ described in \cite[\S~4]{Umemura-1988}.
In fact, the fibers of $\pi$ are proper transforms of bisecant lines to $C$.
We have the following relations in the Chow ring of $\tilde \PP^d$: 
\begin{equation}
\label{eq:Chow}
(H^{*})^d=1,\quad H^*\cdot E^{d-1}=(-1)^d d,\quad E^d=(-1)^d(d-1)(d+2),
\end{equation}
and $(H^*)^{k}\cdot E^{d-k}=0$ for $2\le k<d$.
The relations follow from the projection formula, 
general properties of blowups, and the fact that 
$\deg \NNN_{C/\PP_d}=-K_{\PP^d}\cdot C-2=(d-1)(d+2)$.
\end{notation}

\begin{lemma}
\label{lemma:Link:C5}
In the above notation let $d=5$.
Then there exists 
the following Sarkisov link 
\begin{equation}
\label{eq:Link:C5}
\vcenter{
\xymatrix{
&\tilde{\PP}^5\ar[dl]_{\sigma}\ar@{-->}[rr]^{\chi}\ar[dr]^{\bar\varphi} && \tilde{\PP}^5_+\ar[dr]^{\varphi}\ar[dl]_{\bar\varphi_+}
\\
\PP^5& & \bar{Y} && \PP^3
} }
\end{equation} 
where
$\bar\varphi$ and $\bar\varphi_+$ are flopping contractions,
$\chi$ is the flop in the proper transform $\tilde Z\subset \tilde{\PP}^5$ of $Z$,
and $\varphi$ is a $\PP^2$-bundle.
\end{lemma}

\begin{proof}[Sketch of the proof]
Since $C$ is an intersection of quadrics, the linear system $|2H^*-E|$ is base point free
and defines a morphism $\bar\varphi=\Phi_{|2H^*-E|}:\tilde{\PP}^5\to \bar Y\subset \PP^{9}$, where $\bar Y:= \bar\varphi(\tilde{\PP}^5)$.
It follows from \eqref{eq:Chow} that $(2H^*-E)^5=10$, hence $\dim \bar Y=5$, i.e. the morphism $\bar\varphi$ is generically finite.
Moreover, $\bar\varphi$ contracts $\tilde Z$ to a surface so that the fibers are
proper transforms of $2$-secant lines of~$C$.

In coordinates, the curve 
$C\subset \PP^5$ is given by vanishing of the minors of the matrix
\[
\begin{pmatrix}
x_{0}&x_{1}&x_{2}&x_{3}&x_{4}\\
x_{1}&x_{2}&x_{3}&x_{4}&x_{5}
\end{pmatrix}
\]
and the map $\bar\varphi\comp \sigma^{-1}:\PP^5 \dashrightarrow \PP^9$ 
is given by the same minors. The equations of the image are
\[
\begin{array}{l}
y_{5}y_{7}-y_{4}y_{8}+y_{2}y_{9}= y_{5}y_{6}-y_{3}y_{8}+y_{1}y_{9}
= y_{2}y_{3}-y_{1}y_{4}+y_{0}y_{5}=
\\
y_{4}y_{6}-y_{3}y_{7}+y_{0}y_{9}= y_{2}y_{6}-y_{1}y_{7}+y_{0}y_{8}=
y_{4}^{2}-y_{2}y_{5}-y_{3}y_{5}-y_{2}y_{7}+y_{1}y_{8}-y_{0}y_{9}=0.
\end{array}
\]
One can check that $\bar Y=\bar\varphi(\tilde{\PP}^5)$ is a variety of degree 
$10$,
the morphism $\tilde{\PP}^5\to \PP^{9}$ is birational onto its image, and $\bar 
Y$ is singular along $F:=\bar\varphi(\tilde Z)$ which is the Veronese surface $\upsilon_3(\PP^2)\subset \PP^5$.

Since $-K_{\tilde \PP^5}=3(2H^*-E)$,
the morphism $\bar\varphi$ is crepant, hence it is a flopping contraction.
Then there exists a flop $\tilde \PP^5 \dashrightarrow \tilde \PP^5_+$ and $K$-negative
extremal Mori contraction $\varphi: \tilde \PP^5_+\to Y_+$.
Let $H^*_+:=\chi_* H^*$ and $E_+:=\chi_* E$.
Since $\uprho(\tilde{\PP}^5_+)=2$ and $-K_{\tilde{\PP}^5_+}$ is nef, there exist a 
Mori contraction $\varphi:\tilde{\PP}^5_+\to Y_+$. 
Now take two general quadrics $Q_1,\, Q_2\subset \PP^5$ containing $C$ and let $X:=Q_1\cap Q_2$.
Then $X$ is a smooth del Pezzo threefold. Let $\tilde X\subset \tilde \PP^5$
and $\tilde X_+\subset \tilde \PP^5_+$ be its proper transforms. 
Note that $\tilde X$ is the intersection of two general members of the base point free linear system $|2H^*-E|$.
By Bertini's theorem $\tilde X$ is smooth. The same arguments show that $\tilde X_+$ is smooth either.
Then the induced morphism $\sigma_X: \tilde X\to X$ is the blowup of $C$
and the induced map $\chi_X: \tilde X \dashrightarrow \tilde X_+$ is a flop. 
Hence, there exist a 
Mori contraction $\varphi':\tilde{X}_+\to X'$. 
We obtain the following diagram,
where $X_+:=\varphi(\tilde X_+)$:
\begin{equation*}
\vcenter{
\xymatrix{
&\tilde \PP^5 
\ar[dr]^{\bar\varphi}
\ar[dl]_{\sigma}
\ar@{-->}[rr]^{\chi}
&&\tilde \PP^5_+
\ar[dl]_{\bar\varphi_+}
\ar[dr]^{\varphi}
\\
\PP^5
&
\tilde X
\ar[dl]_{\sigma_X}
\ar@{-->}@/_1.1em/[rr]_{\chi_{X}} 
\ar@{} !<-0pt,0pt>;[u]!<-0pt,0pt> |-*[@]{\subset}
&
\bar Y
&
\tilde X_+
\ar[d]_{\varphi'}
\ar[dr]^{\varphi_X} 
\ar@{} !<-0pt,0pt>;[u]!<-0pt,0pt> |-*[@]{\subset}
&Y_+
\\
X 
\ar@{} !<-0pt,0pt>;[u]!<-0pt,0pt> |-*[@]{\subset} 
&&&X' &X_{+}
\ar@{} !<-0pt,0pt>;[u]!<-0pt,0pt> |-*[@]{\subset}
}}
\end{equation*}
Then the sequence $X, \tilde X, \tilde X_+, X'$ forms a three-dimensional Sarkisov link.
Its structure is well-known (see
\cite[Table~3, No.~103]{Cutrone-Marshburn}). Here $X'\simeq \PP^3$ and 
$\varphi'$ is the blowup of a curve of genus $2$ and degree $7$.
Moreover, the morphism $\varphi'$ is given by the linear system $\left|(3H^*_+- 2E_+)\vert_{\tilde X_+} \right|$
and its exceptional divisor is the (unique) member of $\left|(10H^*_+- 7E_+)\vert_{\tilde X_+} \right|$.

Let $\tilde Z_+\subset \tilde \PP^5_+$ be the flopped locus of $\chi$, i.e. the exceptional locus of $\bar \varphi_+$.
Then $\tilde Z_+$ dominates the surface $F$, hence $\tilde Z_+$ is irreducible and $\dim \tilde Z_+=3$.
The flopped locus of $\chi_X$ coincides with $\tilde Z_+\cap \tilde X_+$. Since the variety $\tilde X_+$ is smooth, we conclude that 
$\tilde \PP^5_+$ is smooth along some fibers of $\tilde Z_+\to F$. On the other hand, 
our diagram is equivariant with respect to the action of $\SL_2(\CC)$
and the group $\SL_2(\CC)$ acts on $F$ transitively. This implies that $\tilde \PP^5_+$ is smooth everywhere.
Then by 
\cite{Wisniewski1991a} for any fiber $\Lambda$ of $\varphi$ we have
\begin{equation}
\label{eq:wis}
\dim \Exc(\varphi)+\dim \Lambda\ge 7.
\end{equation}
Note that the variety $Z$ is an intersection of cubics in $\PP^5$
and these cubics are singular along $C$, hence $\Bs |3H^* -2E|=\tilde Z$
and so $\Bs |3H^*_+ -2E_+|\subset \tilde Z_+$. If the divisor $3H^*_+ -2E_+$ is not nef, then 
it must be negative on the fibers of $\varphi$. In this case, $\varphi$ must be a small birational contraction
whose exceptional locus in contained in $\tilde Z_+$. 
The contradiction shows that $3H^*_+ -2E_+$ is nef and so $3H^*_+ -2E_+\sim \varphi^* A$, where $A$ is an ample divisor on
$Y_+$.

Assume that the morphism $\varphi$ is birational. Then by \eqref{eq:wis} it contracts a divisor, say $D$, and
the fibers of $\varphi_D: D\to \varphi(D)$ have dimension $\ge 3$.
Therefore, $D\cap \tilde X_+$ is the exceptional divisor of $\varphi'$.
We can write $K_{\tilde \PP^5_+}=\varphi ^* K_{Y_+}+ a D$ and $-K_{Y_+}\sim bA$ for some $a, b>0$. 
Taking the description of the three-dimensional link into account, we obtain
$D\sim 10H^*_+- 7E_+$ and so 
\[
-(6 H^*_+-3E_+) = -b(3H^*_+ -2E_+)+ a (10H^*_+- 7E_+).
\]
The solution is $(a,b)=(3,12)$, hence $Y_+$ is a Fano fivefolds with terminal singularities whose 
anticanonical class has the form $-K_{Y_+}\sim 12A$, where $A$ is a Cartier divisor.
This is impossible (see e.g. \cite[Corollary~2.1.13]{IP99}).

Therefore, $\varphi$ is a fibration and, again by \eqref{eq:wis}, its fibers have dimension $\ge 2$.
Then $Y_+=X_+\simeq \PP^3$.
The general fiber $\Lambda$ is a smooth del Pezzo surface with 
$-K_{\Lambda}=-K_{\tilde \PP^5_+}\vert _\Lambda=3(2H^*_+-E_+)\vert _\Lambda$, hence $\Lambda\simeq \PP^2$.
As an $\SL_2(\CC)$-variety, the base $Y_+=\PP^3$ has the form $\PP(V)$, where $V$ is a \textit{faithful} four-dimensional 
representation of $\SL_2(\CC)$. From this, one can deduce that there are no $\SL_2(\CC)$-fixed points on $Y_+$.
Now we claim that the morphism $\varphi$ is flat. Indeed, since $\uprho(\tilde \PP^5_+/Y_+)=1$, there are 
at most a finite number of fibers $\Lambda_k$ of dimension $>2$. But then the points $\varphi(\Lambda_k)$ 
must be fixed by the action of $\SL_2(\CC)$, a contradiction. 
Thus \textit{any} fiber $\Lambda$ of $\varphi$ is two-dimensional and $(2H^*_+-E_+)^2\cdot \Lambda=1$.
This implies that $\Lambda\simeq \PP^2$. 
\end{proof}
The following observation was pointed out to us by A.~Kuznetsov.
\begin{sremark}
In coordinates, the variety $Z$ is given by four cubic equations which are the $3\times 3$-minors of the matrix
\[
\begin{pmatrix}
x_0 & x_1 & x_2 & x_3 \\
x_1 & x_2 & x_3 & x_4 \\
x_2 & x_3 & x_4 & x_5
\end{pmatrix}
\]
and these cubics are singular along $C$. 
The map $\varphi\comp \chi \comp \sigma^{-1}:\PP^5 \dashrightarrow \PP^3$ 
is given by the these minors. The fibers of this map are $3$-secant planes of $C\subset \PP^5$.
The restriction of $\chi \comp \sigma^{-1}$ to a general $3$-secant plane is the quadratic Cremona 
transformation.
\end{sremark}

\begin{proof}[Proof of Theorem \xref{prop:V5-V6}]
We have $V_6=\Sy^5(V_2)$, hence the action of $\Alt_5$ on $\PP(V_6)$ is induced
by an action on a Veronese curve $C=C_5\subset \PP(V_6)$. Then we can apply 
the link \eqref{eq:Link:C5} with center $C$ which must be $\Alt_5$-equivariant.
Thus $\PP(V_6)/ \Alt_5$ is birationally equivalent to $\tilde \PP^5_+/\Alt_5$.
Then we can proceed as in the proof of Theorem~\ref{thm:C4/A5} since the anticanonical class of
$\tilde \PP^5_+/\Alt_5$ is divisible by $3$.
\end{proof}

\section{Proof of Theorem~\xref{thm:main}}
\label{sect:V5}
It remains to discuss the rationality of quotient $\PP(V_5)/\Alt_5$.
In this case we can prove only a weaker result:
\begin{proposition}
\label{prop:V5}
The variety $\PP(V_5)/\Alt_5\times \PP^1$ is rational. 
In particular, $V_5/\Alt_5$ is rational as well.
\end{proposition}
\begin{proof}
There exists the following $\Alt_5$-equivariant Sarkisov link (see e.g. \cite[Proposition~9]{Shepherd-Barron-S5}:
\[
\vcenter{
\xymatrix{
& \widetilde{\PP(V_3\oplus V_3')}\ar[dl]_{\sigma}\ar[dr]^{\varphi}& 
\\
\PP(V_3\oplus V_3')\ar@{-->}[rr] && \PP(V_5)
}}
\]
where $\sigma$ is the blowup of the quintic del Pezzo surface $S\subset \PP(V_3\oplus V_3')$ (see Remark~\ref{rem:S5-grA5}) and $\varphi$ is a generically $\PP^1$-bundle.
Now we see that $\widetilde{\PP(V_3\oplus V_3')}/\Alt_5$
is birational to $\PP(V_5)/\Alt_5\times \PP^1$ (cf. \cite[Proof of Proposition~9]{Shepherd-Barron-S5}). 
Indeed, the canonical class of the variety $\widetilde{\PP(V_3\oplus V_3')}$ is divisible by $2$ and 
the action of $\Alt_5$ on it is free in codimension one. Hence the anticanonical class of the quotient 
$\widetilde{\PP(V_3\oplus V_3')}/\Alt_5$ is divisible by $2$ as well (at least modulo torsion).
This defines a birational section of the rational curve fibration $\widetilde{\PP(V_3\oplus V_3')}/\Alt_5 \to \PP(V_5)/\Alt_5$.
Then this fibration must be locally trivial in Zariski topology.
On the other hand, the projection 
$\PP(V_3\oplus V_3') \dashrightarrow \PP(V_3)$ gives us a birational equivalence 
$\PP(V_3\oplus V_3')/\Alt_5$ and $\PP(V_3)/\Alt_5\times \PP(V_3')$
(see \cite[Corollary 3.12]{colliotthelene-Sansuc-2005}). Hence $\PP(V_3\oplus V_3')/\Alt_5$ is rational
and so is $\widetilde{\PP(V_3\oplus V_3')}/\Alt_5$.
\end{proof}

\begin{proof}[Proof of Theorem~\xref{thm:main}]
Denote by $G$ the image of $\tilde \Alt_5$ in $\GL(V)$. Thus
the action of $G$ on $V$ is faithful and $G$ is isomorphic either to $\tilde \Alt_5$ or $\Alt_5$.
Let $\tilde V\to V$ be the blowup of the origin.
Then $\tilde V$ has a structure of $G$-equivariant $\Aff^1$-bundle $\tilde V\to \PP(V)$
admitting an invariant section. Therefore, 
$V/G$ is birational to $\PP(V)/G\times \PP^1$.
Now we prove the statement by induction on the dimension of~$V$.
If $V$ is irreducible, then $\PP(V)/G \times \PP^1$ is rational
by Theorem~\ref{th-SL25}, \ref{thm:C4/A5}, \ref{prop:V5-V6} and Proposition~\ref{prop:V5}.
Thus we may assume that there is a nontrivial decomposition $V=V'\oplus V''$, where the action of $G$ on $V''$ is faithful. 
Then $V/G$ is birational to $V''/G\times V'$ and $V''/G$ is rational by the inductive hypothesis.
\end{proof}

\begin{sremark}
The above proof can be modified to show that the variety $\PP(V)/\tilde \Alt_5$ is rational with 
one possible exception $V\simeq V_5$. 
\end{sremark}

In conclusion, we consider one more interesting variety related to icosahedron.
\subsection{Mukai--Unemura threefold}
Recall that the Poincar\'e homology 
$3$-sphere can be constructed as the quotient $\SO_3(\RR)/\Alt_5$.
One can consider the complex counterpart of it: the quotient $\SL_2(\CC)/\tilde \Alt_5$,
where the action is via the left multiplication. Thus variety is not proper but it 
has a nice smooth compactification called \textit{Mukai--Unemura threefold} \cite{Mukai-Umemura-1983}:
\[
X^{\mathrm{mu}}:=\overline{\SL_2(\CC) \cdot h_{12}}\subset \PP(M_{12}),
\]
where $M_{12}$ is the space of binary forms of degree $12$ and $h_{12}=h_{12}(x_1,x_2)\in M_{12}$
is the invariant of $\tilde\Alt_5$ which has the form 
\[
h_{12}=x_1x_2(x_1^{10} +11 x_1^4x_2^5+ x_2^{10}).
\]
The variety $X^{\mathrm{mu}}$ is a distinguished Fano threefold of Picard rank $1$ and degree $22$.
It satisfies many remarkable properties (see \cite{Mukai-Umemura-1983}, \cite{P:90aut:en}, \cite{Prokhorov-1990b}, \cite{Furushima1992b},
\cite{Donaldson:V22}). 
By the construction, $X^{\mathrm{mu}}$ is quasi-homogeneous, i.e. it admits an $\SL_2(\CC)$-action 
with the open orbit $\SL_2(\CC) \cdot h_{12}$ and the complement $X^{\mathrm{mu}}\setminus \SL_2(\CC) \cdot h_{12}$
is an irreducible divisor. This implies immediately the rationality of $X^{\mathrm{mu}}$ (see e.g. \cite[Lemma~1.15]{Mukai-Umemura-1983}).
Moreover, we have the following.

\begin{stheorem}[\cite{Kolpakov-Prokhorov-1992}, {\cite[Theorem~2.3]{P:Inv10}}]
Let $G\subset \SL_2(\CC)$ be a finite subgroup. If $G$ has a fixed point on $X^{\mathrm{mu}}$, then $X^{\mathrm{mu}}/G$
is rational.
\end{stheorem}
It is easy to see that a subgroup $G\subset \SL_2(\CC)$ has no fixed points only if $G$
is a binary octahedral group, i.e. $G\simeq \tilde \Sym_4$. 
In this case we have a weaker result:
\begin{stheorem}[cf. {\cite[Theorem~2.3]{P:Inv10}}]
The variety $ X^{\mathrm{mu}}/\Sym_4\times \PP^2$ is rational.
\end{stheorem}

\appendix 
\section{Tables}

\renewcommand{\arraystretch}{1.2}
\setlength{\tabcolsep}{0.8em}
\begin{longtable}{ll|l|p{0.2\textwidth}|l|l| l|l}
\caption{Non-trivial proper subgroups of $\Sym_5$} 
\label{table:subgroupsS5}
\\
\\\hlineB{2.5}
&$|H|$ & $H\simeq $ &generators &description&$\subset \Alt_5$?&\# & $\N(H)$ 
\\\hlineB{2.5}
\endhead
\hlineB{2.5}
\endfoot
\hlineB{2.5}
\endlastfoot
\caption{Non-trivial proper subgroups of $\Sym_5$}\\\hlineB{2.5}
&$|H|$ & $H\simeq $ &generators &description&$\subset \Alt_5$?&\# & $\N(H)$
\endfirsthead\hlineB{2.5}

\nr\label{sub:2:tr}&2 & $\mumu_2$ &$(1,2)$ &$\Sym_2(1,2)$ &$-$ &10 & \xref{sub:12:dih}
\\
\nr\label{sub:2:dt}&2 & $\mumu_2$ &$(1,2)(3,4)$ & &$+$ &15 & \xref{sub:8}
\\
\nr\label{sub:3}&3 & $\mumu_3$&$(1,2,3)$ &$\Alt_3(1,2,3)$ &$+$ &10 & \xref{sub:12:dih}
\\
\nr\label{sub:4:K}&4 & $\mumu_2 \times 
\mumu_2$ & \ref{sub:2:dt}, $(1,3)(2,4)$ &Klein $4$-group&$+$ &5 &\xref{sub:24}
\\
\nr\label{sub:4:K1}&4 & $\mumu_2 \times \mumu_2$ &$(1,2),\, (3,4)$ &$\Sym_2(1,2)\times 
\Sym_2(3,4)$ & $-$ &15 &\xref{sub:8}
\\
\nr\label{sub:4:c}&4 & $\mumu_4$ &$(1,3,2,4)$ && $-$ &15 & \xref{sub:8}
\\
\nr\label{sub:5}&5 & $\mumu_5$ &$(1,2,3,5,4)$ && $+$ &6 & \ref{sub:20}
\\
\nr\label{sub:6:Sym}&6 & $\Sym_3$ & \ref{sub:3}, $(1,2)$ &$\Sym_3(1,2,3)$ &$-$ &10 
&\xref{sub:12:dih}
\\
\nr\label{sub:6:Sym1}&6 & $\Sym_3$ & \ref{sub:3}, $(1,2)(4,5)$ && 
$+$ &10 & \xref{sub:12:dih}
\\
\nr\label{sub:6:c}&6 & $\mumu_6$ &$(1,2,3)(4,5)$ && $-$ &10 &\ref{sub:12:dih}
\\
\nr\label{sub:8}&8 & $\Dih_4$ &\ref{sub:4:K}, \ref{sub:4:c} &&$-$ &15 &$H$
\\
\nr\label{sub:10}&10 & $\Dih_{5}$ &\ref{sub:5}, $(1,2)(3,4)$ && $+$ &6 
&\ref{sub:20}
\\
\nr\label{sub:12:Alt}&12 & $\Alt_4$ & \ref{sub:3}, \ref{sub:4:K}&$\Alt_4(1,2,3,4)$ &$+$ 
&5 &\ref{sub:24}
\\
\nr\label{sub:12:dih}&12 & $\Dih_{6}$ &\ref{sub:3}, \ref{sub:6:c}&$\Sym_3\times \Sym_2$&$-$ &10 & 
$H$
\\
\nr\label{sub:20}&20 & $\mumu_5\rtimes \mumu_4$ &\ref{sub:10}, $(1,3,2,4)$ && $-$ &6 
& $H$
\\
\nr\label{sub:24}&24 & $\Sym_4$ & \ref{sub:3}, \ref{sub:8}&$\Sym_4(1,2,3,4)$ &$-$ &5 & $H$ 
\\
\nr\label{sub:60}&60 & $\Alt_5$ &\ref{sub:3}, \ref{sub:4:K}, \ref{sub:5}&$\Alt_5(1,2,3,4,5)$ & $+$ &1 & $\Sym_5$ 
\\

\end{longtable}\begin{enumerate}
\item 

\end{enumerate}

\begin{table}[H]
\def\sizee{2.1em}
\caption{Representations of $\Sym_5$}
\label{Tab:repS5}
\begin{tabularx}{1\textwidth}{XV{2.0} 
>{\RaggedLeft}p{\sizee}>{\RaggedLeft} 
p{\sizee}>{\RaggedLeft}p{\sizee}>{\RaggedLeft}p{\sizee}>{\RaggedLeft}
p{\sizee}>{\RaggedLeft}p{\sizee}>{\RaggedLeft}p{\sizee}}
\hlineB{2.5}
&$\mathcal{C}_{1}^{}$&$\mathcal{C}_{2}^{-}$&$\mathcal{C}_{2}^{+}$&$\mathcal{C}_{3}^{}$&$\mathcal{C}_{6}^{}$&$\mathcal{C}_{4}^{}$&$\mathcal{C}_{5}^{}$
\\\hlineB{2.5}
$W_1$&$1$&$1$&$1$&$1$&$1$&$1$&$1$
\\
$W_1'$&$1$&$-1$&$1$&$1$&$-1$&$-1$&$1$
\\
$W_4$&$4$&$2$&$0$&$1$&$-1$&$0$&$-1$
\\
$W_4'=W_4\otimes W_1'$&$4$&$-2$&$0$&$1$&$1$&$0$&$-1$
\\
$W_5$&$5$&$1$&$1$&$-1$&$1$&$-1$&$0$
\\
$W_5'=W_5\otimes W_1'$&$5$&$-1$&$1$&$-1$&$-1$&$1$&$0$
\\
$W_6=\wedge^2 W_4$&$6$&$0$&$-2$&$0$&$0$&$0$&$1$
\\\hlineB{2.5}
\end{tabularx}
\end{table}

\renewcommand{\arraystretch}{1.5}
\renewcommand{\tabcolsep}{0.7em}
\begin{longtable}{lV{2}ccccccccc}
\caption{Representations of $\tilde\Alt_5=\SL_2(\bF_5)$}
\label{tab:repSL25}
\\\hlineB{2.5}
&$\mathcal{C}_{1,1}^{1}$&$\mathcal{C}_{2,1}^{1}$&$\mathcal{C}_{4,2}^{30}$&
$\mathcal{C}_{5,5}^{12}$&$\mathcal{C}_{5,5}^{12}$&$\mathcal{C}_{10,5}^{12}$&$\mathcal{C}_{10,5}^{12}
$&$\mathcal{C}_{6,3}^{20}$&$\mathcal{C}_{3,3}^{20}$
\\\hlineB{2.5}
\endhead
\hlineB{2.5}
\endfoot
\hlineB{2.5}
\endlastfoot
$V_1$&$1$&$1$&$1$&$1$&$1$&$1$&$1$&$1$&$1$
\\
$V_2$&$2$&$-2$&$0$&$\frac{-1+\sqrt 5}2$&$\frac{-1-\sqrt 5}2$&$\frac{1-\sqrt5}2$&$\frac{1+\sqrt 5}2$&$1$&$-1$
\\
$V_2'$&$2$&$-2$&$0$&$\frac{-1-\sqrt 5}2$&$\frac{-1+\sqrt 5}2$&$\frac{1+\sqrt5}2$&$\frac{1-\sqrt 5}2$&$1$&$-1$
\\
$V_3=\Sy^2(V_2)$&$3$&$3$&$ -1$&$\frac{1-\sqrt 5}2$&$\frac{1+\sqrt 5}2$&$\frac{1-\sqrt5}2$&$\frac{1+\sqrt 5}2$&$
0$&$0$
\\
$V_3'=\Sy^2(V_2')$&$3$&$3$&$ -1$&$\frac{1+\sqrt 5}2$&$\frac{1-\sqrt 5}2$&$\frac{1+\sqrt5}2$&$\frac{1-\sqrt 5}2$&$
0$&$0$
\\
$V_4'=V_2\otimes V_2'$&$4$&$4$&$ 0$&$-1$&$-1$&$-1$&$-1$&$1$&$1$
\\
$V_4=\Sy^3(V_2)$&$4$&$-4$&$ 0$&$-1$&$-1$&$1$&$1$&$-1$&$1$
\\
$V_5=\Sy^4(V_2)$&$5$&$5$&$1$&$0$&$0$&$0$&$0$&$-1$&$-1$
\\
$V_6=\Sy^5(V_2)$&$6$&$-6$&$0$&$1$&$1$&$-1$&$-1$&$0$&$0$
\\
\end{longtable}

\end{document}